\documentclass[12pt]{article}
\usepackage{makeidx}
\usepackage{amssymb,amsthm,amsmath,amsfonts,enumerate}
\usepackage{graphicx,subcaption,float,relsize}
\usepackage{url}
\usepackage{hyperref}
\usepackage{rotating}
\usepackage[longnamesfirst]{natbib}
\usepackage{array}
\usepackage[nodisplayskipstretch]{setspace}
\usepackage{lipsum}
\usepackage{caption}
\usepackage{indentfirst}

\hypersetup{
  colorlinks = true, urlcolor = cyan, linkcolor = blue, citecolor = red }
\renewenvironment{abstract}
 {\small
  \begin{center}
  \bfseries \abstractname\vspace{-.0em}\vspace{0pt}
  \end{center}
  \list{}{    \setlength{\leftmargin}{0mm}
    \setlength{\rightmargin}{\leftmargin}  }  \item\relax}
 {\endlist}
\makeatletter
\def\maketag@@@#1{\hbox{\m@th\normalfont\normalsize#1}}
\makeatother

\DeclareMathOperator{\argmax}{argmax}

\allowdisplaybreaks
\newtheorem {theorem}{Theorem}[section]
\newtheorem {assumption}{Assumption}

\newtheorem{lemma}[theorem]{Lemma}

\newtheorem{remark}{Remark}[section]

\allowdisplaybreaks
\begin{document}

\title{A High Dimensional Wild Bootstrap Max-Test for Detecting the Presence
of Significant Predictors}
\author{Jonathan B. Hill\thanks{%
Department of Economics, University of North Carolina, Chapel Hill, North
Carolina, E-mail:\texttt{jbhill@email.unc.edu}; \texttt{\href{https://tarheels.live/jbhill/}%
{https://tarheels.live/jbhill/}}. We kindly acknowledge Min Qian for
providing the R code developed for the ART procedure in \cite%
{McKeague_Qian_2015}.}\medskip \\
Dept. of Economics, University of North Carolina, Chapel Hill}
\date{{\large This draft:} \today
}
\maketitle

\begin{abstract}
We construct a block bootstrap max-test for detecting the presence of
significant predictors in a high dimensional setting, allowing for weakly
dependent and heterogeneous (possibly non-stationary) data. The number of
covariates to be screened may be large $p$ $>>$ $n$, and growing at an
exponential rate, provided $\ln (p)$ $=$ $o(n^{a})$ for some $a$ $>$ $0$
that depends on memory decay and the growth of higher moments. We study the
problem of correlation screening in a high dimensional marginal regression
setting, assuming so-called \textit{physical dependence} in a time series
setting. We entirely sidestep covariance matrix estimation and adaptive
re-sampling by working with a max-statistic over the many computed
parameters. Thus we do not need endogenous selection of the most relevant
predictor index yielding non-uniform asymptotics, nor do we need a
post-estimation Bonferroni correction. The sequence of non-standard
approximating distributions arising from the maximum of an increasing number
of estimators is easily approximated by a multiplier (wild) block bootstrap.
The max-test controls for size well, performs well against various
deviations from the null, including very slight deviations with a weak or
sparse signal. A numerical experiment is performed and an empirical example
with the VIX volatility index is provided.\medskip \newline
\end{abstract}

\setstretch{1.75}

\section{Introduction\label{sec:intro}}

The selection of statistically significant predictors and subsequent
inference is a central challenge in applied sciences, while recent advances
in high dimensional settings pose new challenges \citep[e.g.][]{FanLi2006}.
High dimensional data sets are now common due to the amount of data
available, the sophisticated nature of data collection methods, and the use
of machine learning and related nonparametric methods for estimation of
unknown link functions. This covers many scientific fields, including
economics, finance, engineering, genetics, communication and meteorology,
and other material sciences %
\citep[e.g.][]{BuhlmannVanDeGeer2011,FanLvQi2011,BelloniChernozhukovHansen2014}%
.

Historically multiple testing and related post-estimation p-value bounding
are exploited %
\citep[e.g.][]{BenjaminiHochberg1995,Efron2006,McCloskey2017,McCloskey2020}.
Conversely, \cite{McKeague_Qian_2015} propose a single \textit{Adaptive
Resampling Test} [ART] for detecting the most significant predictor (if
there is one), while controlling for the family-wise error rate [\textit{fwer%
}]. They work in an iid setting, with methods ultimately akin to those
developed in \citet{Andrews_Cheng_2012} and \cite{Cheng2015} for weakly
identified models, and while the number of covariates $p$ $>$ $n$ is allowed
for \textit{some} sample size $n$, $p$ is otherwise \textit{fixed}, thus $%
p/n\rightarrow 0$ yielding eventually a low dimension setting. Ostensibly a
post estimation framework with \textit{fwer} control in a true high
dimensional setting $p/n$ $\rightarrow $ $\infty $ is not yet available.
Related work for independent data with fixed $p$ $>$ $n$ includes a
self-standardized version of the ART based on a parametric bootstrap with
McKeague and Qian's (\citeyear{McKeague_Qian_2015}) limit theory %
\citep{ZhangLaber2015}, linear conditional mean models %
\citep{TangWangBarut2018}, and survival models \citep{Huang_etal2019}.

\cite{ZhangLaber2015} in a numerical study treat $\max_{1\leq i\leq p}|\hat{t%
}_{i}|$ and $\sum_{i=1}^{p}|\hat{t}_{i}|$ with scale-invariant t-statistic $%
\hat{t}_{i}$ (see (\ref{Tt}) below). They allow for autoregressive
covariates, claim $p$ $\rightarrow $ $\infty $ without a mechanism in theory
or in their simulation study, and use a parametric simulation to approximate
the limit distributions, but supporting theory is not provided and size
control is not studied. Although the parametric simulation-based procedure
avoids a tuning parameter for ART and therefore McKeague and Qian's (%
\citeyear{McKeague_Qian_2015}) double-bootstrap, the data generating
process-dependent limit distributions must be known. Moreover, we find a
max-t or ave-t-test never dominates a max or ave-test in many settings: see
Section \ref{sec:sim}, and see \cite{HillMotegi2020}, \cite%
{Hill_2025_maxtest}, and \cite{HillLi2025}. Indeed, compared to an ave-test
in general a max-test yields better size control over $p$, and greater power
over $p$ with very few exceptions \citep[see
also][]{HillMotegi2020}.

We extend the marginal screening framework to a \textit{true} high
dimensional [HD] setting, with an increasing number of covariates $p$ $=$ $%
p_{n}$ $\rightarrow $ $\infty $ under weak dependence and heterogeneity
assumptions, provided $\ln (p)$ $=$ $o(n^{a})$ for some $a$ $>$ $0$. The
magnitude of $a$ depends on memory decay and growth of higher moments. Thus
we allow for time series or cross sections with non-iid data. We do so 
\textit{pre-estimation} by avoiding the endogenous selection of McKeague and
Qian's (\citeyear{McKeague_Qian_2015}) optimal covariate index altogether.
Indeed, it is unnecessary for the purpose of a test. Likewise, we sidestep
sequential testing and therefore \textit{fwer}\ control, contrary to \cite%
{BenjaminiHochberg1995}, \cite{Dudoit_etal2003}, \cite{Efron2006} and
others, and contrary to \cite{McKeague_Qian_2015} who control \textit{fwer}\
with a single test by using a nuisance sequence $\lambda _{n}$. The latter
is selected in practice by using a second bootstrap procedure. We use a
single test where a single dependent bootstrap naturally yields an
asymptotic level $\alpha $ test, without the need for a nuisance sequence.

Working pre-estimation also allows us to sidestep estimating a full HD model
under sparsity. The latter is computationally intensive, may be unstable or
fragile \citep[e.g.][]{Kolesar_etal_2024}, and ultimately potentially 
\textit{illusory}: sparsity may not be appropriate or feasible in many
economic and finance settings \citep[see,
e.g.][]{Giannone_etal2021}.

Reduced dimension regression models are variously called \textit{marginal
regression}, \textit{correlation learning}, and \textit{sure screening} %
\citep[e.g.][]{FanLv2008,Genovese_etal_2012}, and includes \textit{canonical
correlation analysis} \citep[e.g.][]{McKeagueZhang2022}. Using a standard
marginal regression setting with sample $\{x_{t},y_{t}\}_{t=1}^{n}$, \cite%
{McKeague_Qian_2015} regress $y_{t}$ on each covariate $(x_{i,t}$ $:$ $1$ $%
\leq $ $i$ $\leq $ $p)$ one at a time. This yields marginal coefficients 
\begin{equation*}
\hat{\phi}_{i,n}=\widehat{cov}(x_{i},y)/\widehat{var}(x_{i}),
\end{equation*}%
and an endogenously selected max-index $\hat{l}_{n}$ $\in $ $\arg
\max_{1\leq l\leq p}|\hat{\phi}_{l,n}|$ ideally representing the most
informative predictor when unique, and therefore the most informative
coefficient $\hat{\phi}_{\hat{\imath}_{n},n}$. An implicit iid assumption is
imposed in order to study $\hat{\phi}_{\hat{l}_{n},n}$ as a vehicle for
testing $H_{0}$ : $\phi _{0}^{\ast }$ $=$ $0$ where $\phi _{0}^{\ast }$ $=$ $%
\phi _{i^{\ast }}$ with 
\begin{equation*}
\phi _{i}=\frac{cov\left( x_{i},y\right) }{v(x_{i})}\text{ and }i^{\ast }\in %
\argmax_{1\leq l\leq p}\left\vert \phi _{i}\right\vert .
\end{equation*}%
Thus, the hypotheses of interest reduce to%
\begin{eqnarray}
&&H_{0}:cov\left( x_{i},y\right) =0\text{ }\forall i\Leftrightarrow
H_{0}:\phi _{0}^{\ast }=0  \label{H0H1} \\
&&H_{1}:cov\left( x_{i},y\right) \neq 0\text{ for some }i\Leftrightarrow
H_{1}:\phi _{i}\neq 0.  \notag
\end{eqnarray}%
(In a linear regression setting $y_{t}$ $=$ $\delta _{i}$ $+$ $\phi ^{\prime
}x_{t}$ $+$ $\epsilon _{t}$ with $\mathbb{E}\epsilon _{t}x_{t}$ $=$ $0$,
this is equivalent to testing $H_{0}$ $:$ $\phi $ $=$ $0$ where $\phi $ $\in 
$ $\mathbb{R}^{p}$. Cf. \cite{Hill_2025_maxtest}).

See \cite{McKeague_Qian_2015} for discussion, and for historical details on
the complexity of post model selection and the resulting non-standard
asymptotics for $\sqrt{n}(\hat{\phi}_{\hat{l}_{n},n}$ $-$ $\phi _{0}^{\ast
}) $. ART resolves the non-standard asymptotics implicit in $\sqrt{n}\hat{%
\phi}_{\hat{l}_{n},n}$, taking into account the limit distribution of $\sqrt{%
n}(\hat{\phi}_{\hat{l}_{n},n}$ $-$ $\phi _{0}^{\ast })$ exhibits a break
when $\phi _{0}^{\ast }$ $=$ $0$ because $i^{\ast }$\ is not identified, cf. %
\citet{Andrews_Cheng_2012} and \cite{Hill2021}. A subsequent \textit{double}
bootstrap is required in order to set a nuisance sequence, which is valid
only when $p$ is fixed. This incurs a sizeable computation cost compared to
our approach where $p/n$ $\rightarrow $ $\infty $ \textit{is} allowed, as we
will see in the numerical study below.

The endogenously selected post-estimation $\hat{\phi}_{\hat{l}_{n},n}$ is
identically $\max_{1\leq i\leq p}|\hat{\phi}_{i,n}|$ up to a scale constant
in $\{-1,1\}$, where the latter is naturally rid of endogenous selection
through $\hat{l}_{n}$. The only added information $\hat{\phi}_{\hat{l}%
_{n},n} $ provides is the directional impact (coefficient sign) of the most
informative covariate. But that is irrelevant for a test of (\ref{H0H1}).

Indeed, $\max_{1\leq i\leq p}|\sqrt{n}\hat{\phi}_{i,n}|$ can be
characterized in a variety of dependence and heterogeneity settings. We
study max-statistics like 
\begin{equation}
\mathcal{\hat{T}}_{n}:=\max_{1\leq i\leq p}\left\vert \sqrt{n}\hat{\phi}%
_{i,n}\right\vert \text{ or }\max_{1\leq i\leq p}\left\vert \hat{t}%
_{i}\right\vert \text{ where }\hat{t}_{i}:=\frac{\sqrt{n}\hat{\phi}_{i,n}}{%
se_{i,n}}  \label{Tt}
\end{equation}%
as a means for testing $H_{0}$, allowing for dependence and non-stationarity
in the observed sample $\{x_{t},y_{t}\}_{t=1}^{n}$, and HD in the sense $p$ $%
>>$ $n$ and $p$ $\rightarrow $ $\infty $, where $p/n$ $\rightarrow $ $\infty 
$ is possible. We write $se_{i,n}$ to denote a standard error for $\sqrt{n}%
\hat{\phi}_{i,n}$ in a bivariate regression. In the spirit of \cite%
{ZhangLaber2015} we also tackle average statistics like 
\begin{equation*}
\widehat{\mathfrak{A}}_{n}:=\sum_{i=1}^{p}\left\vert \sqrt{n}\hat{\phi}%
_{i,n}\right\vert \text{ or }\widehat{\mathfrak{A}}_{n}:=\sum_{i=1}^{p}\left%
\vert \hat{t}_{i}\right\vert
\end{equation*}%
due to their plausible strength against weak dense alternatives %
\citep[cf.][]{ZhangLaber2015}. Our limit theory for $\mathcal{\hat{T}}_{n}$
trivially extends to $\widehat{\mathfrak{A}}_{n}$. We allow for a large
class of weights other than $1/se_{i,n}$ as in \cite{HillMotegi2020}, \cite%
{Hill_2025_maxtest} and \cite{HillLi2025}. Further, it is a simple
generalization to allow $p$ $\rightarrow $ $p^{\ast }$ for some $p^{\ast }$ $%
\in $ $\mathbb{N}\cup \infty $, but that does not ensure a consistent test
when the number of covariates available diverges in $n$.

By allowing for dependence the covariates may contain, or be entirely made
up of, lagged values of $y_{t}$. In the latter case $\mathcal{\hat{T}}_{n}$
is a max-autocorrelation, a classic subject in the extreme value theory
literature where $y_{t}$ is white noise under $H_{0}$ 
\citep[see][for
references]{HillMotegi2020}. Although our test is parametric, the model used
need not represent the true data generating process. In this broad setting
we also derive the local power of our test under a sequence of $\sqrt{n}%
/(\ln (p))^{2}$-local alternatives with $\ln (p)$ $=$ $o(n^{1/4})$ allowing
for exponential $p/[\mathcal{C}\exp \{\mathcal{K}n^{1/4}\}]$ $\rightarrow $ $%
0$ for any $\mathcal{C},\mathcal{K}$ $>$ $0$.

Historically there is broad interest in endogenous selection of the
\textquotedblleft \textit{most informative}\textquotedblright\ regressor, or
(adaptive) post-model-selection inference. Indeed, selecting an optimal
predictor can be performed forward-stepwise to generate a subsequent
regression model. See, e.g., \cite{LeebPotscher2006} for deep results, and
consult \cite{LaberMurphy2011}, \cite{Berk_etal2013}, \cite%
{McKeague_Qian_2015} and \citet{McCloskey2017,McCloskey2020}\ for further
reading.

The present paper extends work in \cite{Hill_2025_maxtest}, who uses low
dimensional models with a fixed dimensional nuisance covariate to test a $p$%
-dimensional parameter in an iid regression setting where $p$ $>>$ $n$, $p$ $%
\rightarrow $ $\infty $ and $p/n$ $\rightarrow $ $\infty $ are allowed. The
method relates to work in \cite{TangWangBarut2018}, and generalizes the
max-t-test in \cite{ZhangLaber2015}. \cite{Hill_2025_maxtest} imposes an iid
assumption in a linear regression framework, yielding four covariate tail
cases (bounded, sub-Gaussian, sub-exponential, $\mathcal{L}_{r}$-bounded)
and therefore varying bounds on $p$ (exponential or polynomial). Nuisance
covariates, however, are typically ignored in a correlation screening
setting. Further, weak dependence evidently forces us to work in a now
standard sub-exponential tail setting. Weak dependence is a key requirement
if high dimensionality is due to the nuances of large panel models with a
time dimension, including many qualitative dummies, and (multi-way) fixed
effects \citep{AngristHahn2004,BelloniChernozhukovHansen2014,Correia2016}.
As far as we know, time series and panels have not been studied in the
correlation learning or sure screening literature.

Our work is related to the literature on high dimensional covariance
matrices and the related coherence of a random matrix. Weak limit theory in 
\cite{CaiJiang2012} for coherence estimators can be used to study $%
\max_{1\leq i\leq p}\{(\sqrt{n}\hat{\phi}_{i,n})^{2}\}$ under a harsh
assumption of iid joint normality. This will generally not hold in time
series settings in macroeconomics and finance where conditional volatility
arises. They yield asymptotics under sub-exponential $\ln (p)/n$ $%
\rightarrow $ $0$, exponential $\ln (p)/n$ $\rightarrow $ $(0,\infty )$ and
super-exponential cases. See also \cite{Jiang2004}, \cite{LiLiuRosalsky2009}%
, \cite{LiuLinShao2008}, and \cite{ShaoZhou2014} for related results
governing (nearly) HD covariance matrices, ranging from linear $p$ $\simeq $ 
$n$ to sub-exponential $\ln (p)/n^{a}$ $\rightarrow $ $0 $ for $a$ $\leq $ $%
1/3$, as well as necessary and sufficient conditions. We allow for a
significantly larger class of stochastic processes covering weak dependence,
while still retaining a sub-exponential rate $\ln (p)/n^{a}$ $\rightarrow $ $%
0$ for $a$ $\in $ $(0,1)$ that is larger under faster memory decay or slower
moment growth.

The rest of the paper is laid out as follows. Section \ref{sec:test_stat}
presents the test statistic and bootstrap algorithm. In Section \ref%
{sec:mr_phys} we derive fundamental results under physical dependence, and
present a complete bootstrap theory covering local power in Section \ref%
{sec:boot_theory}. Section \ref{sec:sim} contains a numerical study and
concluding remarks are left for Section \ref{sec:conclus}. All figures,
tables and an empirical study involving the VIX volatility index are
relegated to the supplemental material 
\citet[Appendices G and
H]{sm_sig_pred}.\medskip

We assume all random variables exist on a complete measure space $(\Omega ,%
\mathcal{F},\mathbb{P})$ in order to side-step any measurability issues. $%
a.s.$ is $\mathbb{P}$-\textit{almost surely}. $\mathbb{E}$ is the
expectations operator; $\mathbb{E}_{\mathcal{A}}$ is expectations
conditional on $\mathcal{F}$-measurable $\mathcal{A}$. $|x|$ $=$ $%
\sum_{i,j}|x_{i,j}|$ is the $l_{1}$-norm, $|x|_{2}$ $=$ $%
(\sum_{i,j}x_{i,j}^{2})^{1/2}$ is the Euclidean or $l_{2}$ norm; $||\cdot ||$
is the spectral norm; $||\cdot ||_{r}$ denotes the $\mathcal{L}_{r}$-norm.
Write with some abuse of notation $\max_{i,n}$ $:=$ $\limsup_{n\rightarrow
\infty }\max_{1\leq i\leq p}$, $\max_{i,n,t}$ $:=$ $\limsup_{n\rightarrow
\infty }\max_{1\leq i\leq p}\max_{1\leq t\leq n}$, and $\max_{i,t}$ $:=$ $%
\max_{1\leq i\leq p}\max_{1\leq t\leq n}$. $x\wedge y$ $=$ $\min \{x,y\}$; $%
x\vee y$ $=$ $\max \{x,y\}$. $Z$ $\sim $ $F$ states $Z$ has distribution $F$%
. $\mathcal{I}(\cdot )$ is the indicator function.

As long as there is no confusion, we say \textquotedblleft \textit{uniformly}%
" to denote, e.g., $\max_{i,n,t}$ or $\min_{i,n}$, depending on the implicit
case.

\section{Test Statistic and Bootstrap Procedure\label{sec:test_stat}}

\subsection{Test Statistic}

Marginal regression models for sample $\{y_{t},x_{t}\}_{t=1}^{n}$ with
vector-valued $x_{t}$ $=$ $[x_{i,t}]_{i=1}^{p}$ are%
\begin{equation}
y_{t}=\delta _{i}+\phi _{i}x_{i,t}+v_{i,t}=\beta _{i}^{\prime }\mathbf{x}%
_{i,t}+v_{i,t},  \label{mr_model}
\end{equation}%
where $\mathbf{x}_{i,t}$ $=$ $[1,x_{i,t}]^{\prime }$, and unique $\beta _{i}$
$\in $ $\mathbb{R}^{2}$ satisfies $\mathbb{E}\mathbf{x}_{i,t}v_{i,t}$ $=$ $0$
for each model $i$ $=$ $1,...,p$. Model (\ref{mr_model}) may otherwise be
mis-specified in the classic sense $\mathbb{P}(\mathbb{E}_{x_{i,t}}v_{i,t}$ $%
=$ $0)$ $<$ $1$ \citep[e.g.][]{Sawa1978,White1982}. Under $H_{0}$ it follows 
$y_{t}$ $=$ $\delta _{i}$ $+$ $v_{i,t}$ for all $i$, therefore $\delta _{i}$ 
$=$ $\delta $ $:=$ $\mathbb{E}y_{t}$ $\forall i$. Thus, tautologically $%
v_{i,t}$ $=$ $v_{t}$ $\equiv $ $y_{t}$ $-$ $\mathbb{E}y_{t}$ $\forall i$
under $H_{0}$.

Define the least squares estimator $\hat{\beta}_{i,n}$ $:=$ $(\sum_{i=1}^{n}%
\mathbf{x}_{i,t}\mathbf{x}_{i,t}^{\prime })^{-1}$ $\times $ $\sum_{i=1}^{n}%
\mathbf{x}_{i,t}y_{t}$, and, e.g., $\bar{x}_{n}$ $:=$ $1/n%
\sum_{t=1}^{n}x_{i,t}$, hence classically%
\begin{equation*}
\hat{\phi}_{i,n}=\frac{1/n\sum_{t=1}^{n}\left( x_{i,t}-\bar{x}_{i,n}\right)
\left( y_{t}-\bar{y}_{n}\right) }{1/n\sum_{t=1}^{n}\left( x_{i,t}-\bar{x}%
_{i,n}\right) ^{2}}.
\end{equation*}%
The general (weighted) max- and ave-statistics are 
\begin{equation*}
\mathcal{\hat{T}}_{n}=\max_{1\leq i\leq p}\left\vert \mathcal{W}_{i,n}\sqrt{n%
}\hat{\phi}_{i,n}\right\vert \text{ and }\widehat{\mathfrak{A}}%
_{n}=\sum_{i=1}^{p}\left\vert \mathcal{W}_{i,n}\sqrt{n}\hat{\phi}%
_{i,n}\right\vert ,
\end{equation*}%
where $\mathcal{W}_{i,n}$ are possibly stochastic weights satisfying $%
\max_{1\leq i\leq p}|\mathcal{W}_{i,n}$ $-$ $\mathcal{W}_{i}|$ $\overset{p}{%
\rightarrow }$ $0$ for some sequence of finite non-random numbers $\{%
\mathcal{W}_{i}\}_{i=1}^{\infty }$ where $\min_{i}\{\mathcal{W}_{i}\}$ $>$ $%
0 $. Max- and ave-tests use $\mathcal{W}_{i,n}$ $=$ $1$ while max-t and
ave-t-tests use $\mathcal{W}_{i,n}$ $=$ $1/se_{i,n}$. In a pure
autoregressive [AR] or ARX setting where all or some $x_{i,t}$ $=$ $%
y_{t-j_{i}}$ for some $j_{i}$ $\geq $ $1$, non-random weights $\{\mathcal{W}%
_{i,n}\}$ akin to those in \cite{LjungBox1978}, where $\mathcal{W}_{i,n}$ $%
\rightarrow $ $1$, may be used to control for the impact of using lags %
\citep[cf.][]{HillMotegi2020}. In \citet[Appendix
C]{sm_sig_pred} we present a consistent HAC estimator for the standard error 
$se_{i,n}$ and show it satisfies our weight requirements. Empirically,
however, we find $\mathcal{W}_{i,n}$ $=$ $1$ dominates in nearly all cases
across hypotheses and models studied in Section \ref{sec:sim}. See also
Remark \ref{rem:hac} below.

\subsection{Bootstrap}

We consider a dependent multiplier (wild) bootstrap %
\citep[cf.][]{Bose1988,Liu1988,Shao2011_JoE} or parametric wild bootstrap in
order to approximate the sequence of approximating null distributions for $\{%
\mathcal{\hat{T}}_{n}\}_{n\in \mathbb{N}}$. Identical methods and theory
extend to the average $\widehat{\mathfrak{A}}_{n}$. Define compact parameter
spaces $\Phi _{i}$ $\subset $ $\mathbb{R}^{2}$, and assume $0$ and $\phi
_{i} $ are interior points of $\Phi _{i}$. Define regressor covariance
matrices $\mathcal{H}_{i}$ $=$ $\mathbb{E}\mathbf{x}_{i,t}\mathbf{x}%
_{i,t}^{\prime }.$

\subsubsection{Dependent Wild Bootstrap [DWB]}

Lemma \ref{lm:mr_ng_phys}, below, dictates that we need only bootstrap the
asymptotic approximation $\max_{1\leq i\leq p}|[0,1]\mathcal{H}_{i}^{-1}1/%
\sqrt{n}\sum_{i=1}^{n}\mathbf{x}_{i,t}v_{t}|$. Let $b_{n}$ be a pre-set
block size, $1\leq b_{n}<n$, $b_{n}/n^{1-\iota }$ $\rightarrow $ $0$ for
some tiny $\iota $ $>$ $0$. The number of blocks is $\mathcal{N}_{n}$ $=$ $%
[n/b_{n}]$. If independence were known then $b_{n}$ $=$ $1$ can be assumed,
or if finite or $\mathcal{J}$-dependence were known then $b_{n}$ $=$ $%
\mathcal{J}$ $\in $ $\mathbb{N}$. We only look at the weak dependence case
here, however, thus $b_{n}$ $\rightarrow $ $\infty $.

Define index sets $\mathfrak{B}_{s}$ $:=$ $\{(s$ $-$ $1)b_{n}$ $+$ $1,\dots
,sb_{n}\}$ with $s$ $=$ $1,\dots ,\mathcal{N}_{n}$, and $\mathfrak{B}_{%
\mathcal{N}_{n}+1}$ $=$ $\{\mathcal{N}_{n}b_{n},...,n\}$. Generate iid
random numbers $\{\xi _{1},\dots ,\xi _{\mathcal{N}_{n}}\}$ with $\mathbb{E}%
[\xi _{i}]$ $=$ $0$, $\mathbb{E}[\xi _{i}^{2}]$ $=$ $1$, and $\mathbb{E}[\xi
_{i}^{4}]$ $<$ $\infty $, e.g., $\xi _{s}$ $\sim $ $N(0,1)$. Now define the
block multiplier $\eta _{t}$ $=$ $\xi _{s}$ if $t$ $\in $ $\mathfrak{B}_{s}$%
. The DWB test statistic is 
\begin{equation*}
\mathcal{\tilde{T}}_{n}:=\max_{1\leq i\leq p}\left\vert \sqrt{n}\mathcal{W}%
_{i,n}\left[ 0,1\right] \widehat{\mathcal{H}}_{i,n}^{-1}\widehat{\mathcal{%
\tilde{G}}}_{i,n}\right\vert
\end{equation*}%
where%
\begin{equation*}
\widehat{\mathcal{H}}_{i,n}=\frac{1}{n}\sum_{t=1}^{n}\mathbf{x}_{i,t}\mathbf{%
x}_{i,t}^{\prime }\text{ and }\widehat{\mathcal{\tilde{G}}}_{i,n}:=\frac{1}{n%
}\sum_{t=1}^{n}\eta _{t}\left\{ \mathbf{x}_{i,t}\left( y_{t}-\bar{y}%
_{n}\right) -\frac{1}{n}\sum_{r=1}^{n}\mathbf{x}_{i,r}\left( y_{r}-\bar{y}%
_{n}\right) \right\} .
\end{equation*}%
Centering $\mathbf{x}_{i,t}\left( y_{t}-\bar{y}_{n}\right) $ $-$ $%
1/n\sum_{r=1}^{n}\mathbf{x}_{i,r}(y_{r}-\bar{y}_{n})$ is exploited because $%
\mathbf{x}_{i,t}(y_{t}$ $-$ $\mathbb{E}y_{t})$ need not have a zero mean (it
does not under $H_{1}$ for some $i$), while we need to replicate the
sequence of approximating null distributions under any case for a valid test
asymptotically \citep[see,
e.g.,][]{Shao2011_JoE}.

Repeat the above steps $\mathcal{M}$ times, each time drawing a new
independent sample $\{\xi _{s}\}_{s=1}^{_{\mathcal{N}_{n}}}$, yielding $\{%
\mathcal{\tilde{T}}_{n,j}\}_{j=1}^{\mathcal{M}}$. The approximate p-value is 
$\hat{p}_{n,\mathcal{M}}$ $:=$ $1/\mathcal{M}\sum_{j=1}^{\mathcal{M}}%
\mathcal{I}(\mathcal{\tilde{T}}_{n,j}$ $\geq $ $\mathcal{\hat{T}}_{n})$. A
bootstrap test rejects $H_{0}$ at significance level $\alpha $ when $\hat{p}%
_{n,\mathcal{M}}$ $<$ $\alpha $.

\subsubsection{Parametric (Dependent) Wild Bootstrap [PWB]}

Let $\{\eta _{t}\}$ be the same bootstrap draw as above.\ Define residuals
under the null $\hat{v}_{n,t}$ $\equiv $ $y_{t}$ $-$ $\bar{y}_{n},$ and
generate wild bootstrap draws under the null $y_{n,t}^{\ast }$ $\equiv $ $%
\bar{y}_{n}$ $+$ $\hat{v}_{n,t}\eta _{t}$. Now construct regression models $%
y_{n,t}^{\ast }$ $=$ $\tilde{\delta}_{i}$ $+$ $\tilde{\phi}_{i}x_{i,t}$ $+$ $%
\tilde{v}_{n,i,t}$ for each $i$, and let $\widehat{\tilde{\phi}}_{i,n}$ be
the bootstrapped least squares estimator of $\tilde{\phi}_{i}$. The PWB test
statistic is $\mathcal{\tilde{T}}_{n}^{(p)}$ $\equiv $ $\max_{1\leq i\leq
k_{\theta ,n}}|\sqrt{n}\mathcal{W}_{n,i}\widehat{\tilde{\phi}}_{i,n}|$.
Repeat $\mathcal{M}$ times yielding a p-value approximation $\hat{p}_{n,%
\mathcal{M}}^{(p)}$ $:=$ $1/\mathcal{M}\sum_{j=1}^{\mathcal{M}}\mathcal{I}(%
\mathcal{\tilde{T}}_{n,j}^{(p)}$ $\geq $ $\mathcal{\hat{T}}_{n})$.

As presented, where the null hypothesis is imposed and the dependent
bootstrap uses centering, the two methods and their generated draws $\sqrt{n}%
\mathcal{W}_{i,n}\left[ 0,1\right] \widehat{\mathcal{H}}_{i,n}^{-1}\widehat{%
\mathcal{\tilde{G}}}_{i,n}$ and $\sqrt{n}\mathcal{W}_{n,i}\widehat{\tilde{%
\phi}}_{i,n}$ are distinct in finite samples and uniformly asymptotically,
though both yield valid tests asymptotically. See 
\citet[Appendix
D]{sm_sig_pred} for general comparisons.

\section{HD Marginal Screening Under Physical Dependence\label{sec:mr_phys}}

We first describe the dependence framework and then present main results.

\subsection{Physical Dependence}

We work in the setting of \cite{Wu2005}, cf. \cite{WuLin2005}. Let $%
\{\epsilon _{i,t}$ $:$ $1$ $\leq $ $i$ $\leq $ $p,\tilde{\epsilon}%
_{t}\}_{t\in \mathbb{Z}}$ be iid sequences, and assume measurable functions $%
g_{i,t}(\cdot )$ and $\tilde{g}_{t}(\cdot )$\ exist satisfying%
\begin{equation*}
x_{i,t}=g_{i,t}\left( \epsilon _{i,t},\epsilon _{i,t-1},\ldots \right) \text{
and }y_{t}=\tilde{g}_{t}\left( \tilde{\epsilon}_{t},\tilde{\epsilon}%
_{t-1},\ldots \right) .
\end{equation*}%
Thus $g_{i,t}$ and $\tilde{g}_{t}$\ may depend on $t$, allowing for
non-stationarity. Examples abound, including linear and nonlinear processes
like ARMA and threshold models, and random volatility like GARCH models or
stochastic volatility. We assume\ below that $\{x_{i,t},y_{t}\}$ have all
moments finite to expedite asymptotic theory, thus the GARCH case, for
example, would hold when underlying innovations like $\epsilon _{i,t}$ have
bounded support \citep[e.g.][]{Basrak_etal2002}. Moment conditions can be
weakened by case at the expense of more complicated proofs.

Let $\{\epsilon _{i,t}^{\prime },\tilde{\epsilon}_{t}^{\prime }\}$ be
independent copies of $\{\epsilon _{i,t},\tilde{\epsilon}_{t}\}$ and define
coupled processes%
\begin{eqnarray*}
&&x_{i,t}^{\prime }(m):=g_{i,t}\left( \epsilon _{i,t},\ldots ,\epsilon
_{i,t-m+1},\epsilon _{i,t-m}^{\prime },\epsilon _{i,t-m-1},\ldots \right) 
\text{ for }m\geq 0 \\
&&y_{t}^{\prime }(m):=\tilde{g}_{t}\left( \tilde{\epsilon}_{t},\ldots ,%
\tilde{\epsilon}_{t-m+1},\tilde{\epsilon}_{t-m}^{\prime },\tilde{\epsilon}%
_{t-m-1},\ldots \right) .
\end{eqnarray*}%
Now define $\mathcal{L}_{r}$-physical dependence [$\mathcal{L}_{r}$-\textit{%
pd}] coefficients $\{\theta _{i,t}^{(r)},\tilde{\theta}_{t}^{(r)}\}$ and
their aggregations 
\begin{eqnarray*}
&&\theta _{i,t}^{(r)}(m):=\left\Vert x_{i,t}-x_{i,t}^{\prime }(m)\right\Vert
_{r}\text{ with }\Theta _{i,t}^{(r)}:=\sum_{m=0}^{\infty }\theta
_{i,t}^{(r)}(m) \\
&&\tilde{\theta}_{t}^{(r)}(m):=\left\Vert y_{t}-y_{t}^{\prime
}(m)\right\Vert _{r}\text{ with }\tilde{\Theta}_{t}^{(r)}:=\sum_{m=0}^{%
\infty }\tilde{\theta}_{t}^{(r)}(m)
\end{eqnarray*}

A common approach is to decompose coefficients into heterogeneity and memory
components \citep[e.g.][]{McLeish1975,
Andrews1988}. Thus, let%
\begin{equation}
\theta _{i,t}^{(r)}(m)\vee \tilde{\theta}_{t}^{(r)}(m)\leq
d_{i,t}^{(r)}\times \psi _{i,m}  \label{thmth}
\end{equation}%
where \textit{coefficients}\ $\psi _{i,m}$ satisfy $\max_{i,n}\psi _{i,m}$ $%
= $ $O(m^{-\lambda -\iota })$ and $\psi _{i,0}$ $=$ $1$\ for some \textit{%
size} $\lambda $ $\geq $ $1$ depicting memory decay, and \textit{constants} $%
d_{i,t}$ that capture heterogeneity. Such a representation holds for many
processes, including linear and nonlinear structures, nonlinear Markov
chains, (nonlinear) GARCH, and so on, cf. \cite{Wu2005} and others. Indeed, $%
\mathcal{L}_{r}$-\textit{pd} implies and is implied by the mixingale
property (with an asymmetric size implication), and mixingales nest Near
Epoch Dependent and mixing processes \citep[see][]{Hill_2025_mixg}. Thus (%
\ref{thmth}) covers a very broad range of dependent processes.

By construction and Minkowski's inequality $\theta _{i,t}^{(r)}(m)\vee 
\tilde{\theta}_{t}^{(r)}(m)$ $\leq $ $2\{||x_{i,t}||_{r}\vee ||y_{t}||_{r}\}$%
, hence we have an upper bound $d_{i,t}^{(r)}$ $\leq $ $K\{\left\Vert
x_{i,t}\right\Vert _{r}$ $\vee $ $\left\Vert y_{t}\right\Vert _{r}\}$. Then,
given $\lambda $ $\geq $ $1$, the aggregations are (uniformly) bounded: 
\begin{align}
& & \Theta _{n}^{(r)}& :=\max_{i,t}\Theta _{i,t}^{(r)}\leq
K\max_{i,t}\left\Vert x_{i,t}\right\Vert _{r}\left( 1+\sum_{m=1}^{\infty
}m^{-\lambda -\iota }\right) =K\max_{i,t}\left\Vert x_{i,t}\right\Vert _{r}
\label{PHI_z} \\
& & \tilde{\Theta}_{n}^{(r)}& :=\max_{t}\{\tilde{\Theta}_{t}^{(r)}\}\leq
K\max_{t}\left\Vert y_{t}\right\Vert _{r}\left( 1+\sum_{m=1}^{\infty
}m^{-\lambda -\iota }\right) =K\max_{t}\left\Vert y_{t}\right\Vert _{r}. 
\notag
\end{align}

\begin{assumption}
\label{assum:marg-reg} \ \ \medskip \newline
$a$. \textbf{Dependence}: $(x_{i,t},y_{t})$ are second order stationary,
governed by non-degenerate distributions; $\mathcal{L}_{r}$-\emph{pd} for
some $r$ $\geq $ $4$, with size $\lambda $ $\geq $ $2$; and $%
\max_{i,n,t}\{||x_{i,t}||_{r}\vee ||y_{t}||_{r}\}$ $\leq $ $ar^{\mu }$ for
some finite $a>0$ and $\mu $ $\in $ $(1/2,\infty )$.\medskip \newline
$b$. \textbf{Identification}: $\mathbb{E}(y_{t}$ $-$ $\beta _{i}^{\ast
\prime }\mathbf{x}_{i,t})\mathbf{x}_{i,t}$ $=$ $\boldsymbol{0}_{2}$ for each 
$i$ and unique $\beta _{i}^{\ast }$ in the interior of $\mathcal{B}_{i}$.$%
\medskip $\newline
$c.$ \textbf{Non-degeneracy}: $\max_{i,n}1/n\sum_{t=1}^{n}(x_{i,t}$ $-$ $%
\bar{x}_{i,n})^{2}$ $>$ $0$ $a.s.$, $\max_{i,n}\mathbb{E}(x_{i,t}$ $-$ $%
\mathbb{E}x_{i,t})^{2}$ $>$ $0$ and $\lim \inf_{n\rightarrow \infty
}\{\inf_{\varsigma ^{\prime }\varsigma =1}\mathbb{E}(1/\sqrt{n}%
\sum_{t=1}^{n}\varsigma ^{\prime }\mathcal{H}_{i}^{-1}\mathbf{x}%
_{i,t}v_{t}\varsigma )^{2}\}$ $>$ $0$ for each $i$ where $\mathcal{H}_{i}$ $%
:=$ $\mathbb{E}\mathbf{x}_{i,t}\mathbf{x}_{i,t}^{\prime }$.
\end{assumption}

\begin{remark}
\normalfont$\mathcal{L}_{r}$-\emph{pd} is imposed unit-wise and does not
restrict dependence \emph{across} units $i$. We must restrict cross-unit
dependence, however, for a bootstrap theory in Section \ref{sec:boot_theory}%
. We generally only need $r$ $\geq $ $4$, but a HAC standard error weight
requires $r$ $\geq $ $8$ to ensure products like $x_{i,t}^{2}x_{i,t-l}^{2}$
are $\mathcal{L}_{r/4}$-\emph{pd} for $r/4$ $\geq $ $2$, the minimal
requirement for high dimensional test consistency. See 
\citet[Appendix
C]{sm_sig_pred}.
\end{remark}

\begin{remark}
\normalfont The moment exponent $\mu $ $>$ $1/2$\ in ($a$) allows for moment
growth that is too large to support sub-exponentiality 
\citep[cf.][Proposition
2.7.1]{Vershynin2018}. Thinner tails (smaller $\mu $) and faster memory
decay (larger $\lambda $) will imply a larger upper bound on $p$ growth. The
lower bound $\mu $ $>$ $1/2$ simplifies technical arguments at little cost
since, e.g., uniformly $a.s.$ bounded $\{x_{i,t},y_{t}\}$ satisfy $%
\max_{i,n,t}\{||x_{i,t}||_{r}\vee ||y_{t}||_{r}\}$ $\leq $ $K$ $\leq $ $%
ar^{1/2}$ for any $a$ $\geq $ $K$ and all $r$ $\geq $ $1$.
\end{remark}

\begin{remark}
\normalfont($b$) is a standard identification condition. ($c$) implies $%
(1/n\sum_{t=1}^{n}\mathbf{x}_{i,t}\mathbf{x}_{i,t}^{\prime })^{-1}$ and $%
\mathcal{H}_{i}^{-1}$\ exist uniformly in $i$ (e.g. $\lim \inf_{n\rightarrow
\infty }\min_{1\leq i\leq p}\inf_{\lambda ^{\prime }\lambda =1}\lambda 
\mathcal{H}_{i}\lambda $ $>$ $0$). Non-degeneracy $\inf_{\lambda ^{\prime
}\lambda =1}\mathbb{E}(1/\sqrt{n}\sum_{t=1}^{n}\lambda ^{\prime }\mathcal{H}%
_{i}^{-1}\mathbf{x}_{i,t}v_{t}\lambda )^{2}$ $>$ $0$ in ($c$)\ is standard
under weak dependence, and required for ensuring non-degenerate asymptotics
when standardization is not necessarily exploited as done here. It holds
when $\boldsymbol{Z}_{n,i}$ $:=$ $[x_{i,1}v_{1},$ $...,$ $%
x_{i,n}v_{n}]^{\prime }$ satisfies a customary positive definiteness
property, $\lim \inf_{n\rightarrow \infty }\inf_{\lambda ^{\prime }\lambda
=1}\mathbb{E}(\lambda ^{\prime }\boldsymbol{Z}_{n,i})^{2}$ $>$ $0$.
\end{remark}

\subsection{Main Results}

We study $\max_{1\leq i\leq p}|\sqrt{n}\mathcal{W}_{i,n}\hat{\phi}_{i,n}|$;
identical arguments extend to $\sum_{i=1}^{p}|\sqrt{n}\mathcal{W}_{i,n}\hat{%
\phi}_{i,n}|$. We need two preliminary results establishing a HD first-order
non-Gaussian approximation and a HD Gaussian approximation. First, the
non-Gaussian approximation.

\begin{lemma}
\label{lm:mr_ng_phys}Under Assumption \ref{assum:marg-reg} with size $%
\lambda $ $\geq $ $1$, $\ln (p)$ $=$ $o(n^{1/4})$, and $H_{0}$ we have%
\begin{equation*}
\left\vert \max_{1\leq i\leq p}\left\vert \sqrt{n}\hat{\phi}%
_{i,n}\right\vert -\max_{1\leq i\leq p}\left\vert \frac{1}{\sqrt{n}}%
\sum\nolimits_{t=1}^{n}\left[ 0,1\right] \mathcal{H}_{i}^{-1}\mathbf{x}%
_{i,t}v_{t}\right\vert \right\vert =O_{p}\left( \left( \ln (p)\right) ^{2}/%
\sqrt{n}\right) =o_{p}(1).
\end{equation*}
\end{lemma}

Next, a Gaussian approximation for $\max_{1\leq i\leq p}|1/\sqrt{n}%
\sum_{t=1}^{n}\left[ 0,1\right] \mathcal{H}_{i}^{-1}\mathbf{x}_{i,t}v_{t}|$.
Define a long-run variance, which is uniformly bounded from below by
Assumption \ref{assum:marg-reg}.c: 
\begin{equation*}
\sigma _{n}^{2}(i):=\mathbb{E}\left( \frac{1}{\sqrt{n}}\sum%
\nolimits_{t=1}^{n}\left[ 0,1\right] \mathcal{H}_{i}^{-1}\mathbf{x}%
_{i,t}v_{t}\right) ^{2}>0\text{ uniformly in }i.
\end{equation*}%
In the proof of the Lemma \ref{lm:mr_gauss_phys} we also show $\max_{1\leq
i\leq p}\sigma _{n}^{2}(i)$ $=$ $O(1)$: it is uniformly bounded from above
under $\mathcal{L}_{r}$-\emph{pd. }Now let $\{\mathcal{Z}_{n}(i)$ $:$ $1$ $%
\leq $ $i$ $\leq $ $p\}_{n\geq 1}$ be a Gaussian array, $\mathcal{Z}_{n}(i)$ 
$\sim $ $\mathcal{N}(0,\sigma _{n}^{2}(i))$. Define the Kolmogorov distance
for the Gaussian approximation,%
\begin{equation}
\rho _{n}:=\sup_{c\geq 0}\left\vert \mathbb{P}\left( \max_{1\leq i\leq
p}\left\vert \frac{1}{\sqrt{n}}\sum\nolimits_{t=1}^{n}\left[ 0,1\right] 
\mathcal{H}_{i}^{-1}\mathbf{x}_{i,t}v_{t}\right\vert >c\right) -\mathbb{P}%
\left( \left\vert \max_{1\leq i\leq p}\left\vert \mathcal{Z}%
_{n}(i)\right\vert \right\vert >c\right) \right\vert .  \label{rho_n}
\end{equation}

\begin{lemma}
\label{lm:mr_gauss_phys}Under Assumption \ref{assum:marg-reg} with size $%
\lambda $ $\geq $ $2$, $\rho _{n}$ $\rightarrow $ $0$ provided $\ln (p)$ $=$ 
$o(n^{\frac{\lambda }{8+2\lambda }\frac{1}{(7/6)\vee (1+\mu )}})$.
\end{lemma}

\begin{remark}
\normalfont The bound on $p$ naturally depends on tail conditions (i.e. $\mu 
$) and memory decay (i.e. $\lambda $). The upper bound increases
monotonically in $\lambda $ and $1/\mu $: weaker dependence ($\lambda $ $%
\uparrow $) and/or thinner tails ($\mu $ $\downarrow $) permit a
larger/faster dimension growth. For example, as $\lambda $ $\rightarrow $ $%
\infty $ (geometric memory, e.g. independence) and $\mu $ $=$ $1/6$
(sub-exponential tails) we have $\ln (p)$ $=$ $o(n^{3/7})$. Under weak
memory decay and \textquotedblleft heavier" tails, however, the bound on $p$
remains exponential, but precipitously shrinks. Suppose $\lambda $ $=$ $4$
(hyperbolic memory) and $\mu $ $=$ $2$ (non-sub-exponential tails). Then $%
\ln (p)$ $=$ $o(n^{1/12})$, which shrinks to $\ln (p)$ $=$ $o(n^{1/36})$
when $\lambda $ $=$ $2$ and $\mu $ $=$ $5$.
\end{remark}

Lemmas \ref{lm:mr_ng_phys} and \ref{lm:mr_gauss_phys} yield the following
main result. The complexity of the upper bound on $p$\ arises from parsing $%
\ln (p)$ $=$ $o(n^{1/4})$ and $\ln (p)$ $=$ $o(n^{\frac{\lambda }{8+2\lambda 
}\frac{1}{(7/6)\vee (1+\mu )}})$ in Lemmas \ref{lm:mr_ng_phys} and \ref%
{lm:mr_gauss_phys}, and accounting for the rate $\max_{1\leq i\leq p}|%
\mathcal{W}_{i,n}$ $-$ $\mathcal{W}_{i}|$ $\overset{p}{\rightarrow }$ $0$.
Define%
\begin{equation*}
\rho _{n}^{\prime }:=\sup_{c\geq 0}\left\vert \mathbb{P}\left( \max_{1\leq
i\leq p}\left\vert \sqrt{n}\mathcal{W}_{i,n}\hat{\phi}_{i,n}\right\vert
>c\right) -\mathbb{P}\left( \left\vert \max_{1\leq i\leq p}\left\vert 
\mathcal{Z}_{n}(i)\right\vert \right\vert >c\right) \right\vert .
\end{equation*}

\begin{theorem}
\label{thm:mr_gauss_phys}Let Assumption \ref{assum:marg-reg} hold with size $%
\lambda $ $\geq $ $2$, and let $H_{0}$ hold. Let $\{\mathcal{W}_{i,n}\}$ be
positive weights satisfying for uniformly positive and bounded
non-stochastic $\{\mathcal{W}_{i}\}$:%
\begin{equation}
\max_{1\leq i\leq p}\left\vert \mathcal{W}_{i,n}-\mathcal{W}_{i}\right\vert
=o_{p}\left( 1/(\ln (p))^{1+\kappa }\right) \text{ for some }\kappa >1/2.
\label{WnW}
\end{equation}%
Let $\{\mathcal{Z}_{n}(i)\}$ be the Lemma \ref{lm:mr_gauss_phys} Gaussian
process, $\mathcal{Z}_{n}(i)$ $\sim $ $\mathcal{N}(0,\sigma _{n}^{2}(i))$.
Then $\rho _{n}^{\prime }$ $\rightarrow $ $0$ provided $\ln (p)$ $=$ $%
o(n^{s(\mu ,\kappa ,\lambda )})$ with $s(\mu ,\kappa ,\lambda )$ $:=$ $\frac{%
1}{4(1+\kappa )}$ $\wedge $ $\frac{\lambda }{8+2\lambda }\frac{1}{(7/6)\vee
(1+\mu )}$.
\end{theorem}

\begin{remark}
\normalfont It is easy to verify by case:%
\begin{equation}
\begin{array}{ll}
\text{if }\mu \in (0,1/6] & \text{then }s(\mu ,\kappa ,\lambda )=\left\{ 
\begin{array}{ll}
\frac{1}{4\left( 1+\kappa \right) } & \text{if }\lambda \geq \frac{56}{%
12\left( 1+\kappa \right) +7}:=\lambda _{1}^{\ast }(\kappa ) \\ 
\frac{\lambda }{8+2\lambda }\frac{1}{(7/6)\vee (1+\mu )} & \text{if }\lambda
<\frac{56}{12\left( 1+\kappa \right) +7}%
\end{array}%
\right. \\ 
\text{if }\mu \in (1/6,1) & \text{then }s(\mu ,\kappa ,\lambda )=\left\{ 
\begin{array}{ll}
\frac{1}{4\left( 1+\kappa \right) } & \text{if }\lambda \geq \frac{4\left(
1+\mu \right) }{2\left( 1+\kappa \right) -\left( 1+\mu \right) }:=\lambda
_{2}^{\ast }(\mu ,\kappa ) \\ 
\frac{\lambda }{8+2\lambda }\frac{1}{(7/6)\vee (1+\mu )} & \text{if }\lambda
<\frac{4\left( 1+\mu \right) }{2\left( 1+\kappa \right) -\left( 1+\mu
\right) }%
\end{array}%
\right. \\ 
\text{if }\mu \geq 1 & \text{then }s(\mu ,\kappa ,\lambda )=\frac{\lambda }{%
8+2\lambda }\frac{1}{1+\mu }.%
\end{array}
\label{s(b,lam)}
\end{equation}
\end{remark}

\begin{remark}
\normalfont Limiting $\mathcal{W}_{i}$ $>$ $0$ ensures a consistent test.
The weights must satisfy (\ref{WnW}) in order for possibly stochastic $%
\mathcal{W}_{i,n}$ not to affect the limit distribution of $\max_{1\leq
i\leq p}|\sqrt{n}\mathcal{W}_{i,n}\hat{\phi}_{i,n}|$. The value $\kappa $ $>$
$1/2$ is set by the practitioner by first choosing $p$ to satisfy $\ln (p)$ $%
=$ $o(n^{s(\mu ,\kappa ,\lambda )})$. Then, for example, non-stochastic
weights $\mathcal{W}_{i,n}$ $=$ $\mathcal{W}_{i}$ or HAC-based weights $%
\mathcal{W}_{i,n}$ necessarily satisfy (\ref{WnW}): see 
\citet[Appendix
C]{sm_sig_pred}. We can always set $\kappa $ arbitrarily close to $1/2$ for
non-stochastic or HAC weights, ensuring the greatest upper bound for $p$
when $\mu $ $\in $ $(0,1)$.
\end{remark}

\begin{remark}
\normalfont If $\mu $ $\leq $ $1/6$ then tails are sub-exponential %
\citep[cf.][Proposition 2.7.1]{Vershynin2018} and $\ln (p)$ $=$ $%
o(n^{1/[4(1+\kappa )]})$ if memory decay is fast enough $\lambda $ $\geq $ $%
\lambda _{1}^{\ast }(\kappa )$. If, e.g., $\mu $ $<$ $1/6$, $\kappa $ $=$ $%
0.51$, and $\lambda $ $=$ $4$ then $\ln (p)$ $=$ $o(n^{3/14})$ a slower
upper bound. In range $\mu $ $\in $ $(1/6,1)$ tails are sub-exponential but $%
\ln (p)$ $=$ $o(n^{1/[4(1+\kappa )]})$ when $\lambda $ $\geq $ $\lambda
_{2}^{\ast }(\mu ,\kappa )$ $\searrow $ $28/[12(1$ $+$ $\kappa )$ $-$ $7]$
as $\mu $ $\searrow $ $1/6$: thinner tails allow for slower memory decay, a
classic trade-off.
\end{remark}

\begin{remark}
\normalfont\label{rem:hac}A max-t-statistic $\max_{1\leq i\leq p}|\sqrt{n}%
\hat{\phi}_{i,n}/se_{i,n}|$ may be preferred in practice to control for
estimator dispersion and to ensure scale invariance. This may be helpful
when regressors have widely different variances, or when the error is
conditionally heteroscedastic. Experimental evidence, however, suggests the
added estimation error caused by the standard least squares $se_{i,n}^{(1)}$
or HAC $se_{i,n}^{(2)}$\ may impair test performance with over rejection of
the null, even in an iid setting, in particular when a HAC estimator is used %
\citep[see][]{HillMotegi2020,Hill_2025_maxtest,HillLi2025}. Only the latter 
\textit{guarantees} $\sqrt{n}\hat{\phi}_{i,n}/se_{i,n}^{(2)}$ $\overset{d}{%
\rightarrow }$ $\mathcal{N}(0,1)$ under $H_{0}$ and can be used for robust
high dimensional interval construction, while the former reduces complexity
and therefore computation time, and still yields a valid and consistent
test. The former $se_{i,n}^{(1)}$ is the classic least squares standard
error which is valid as a consistent estimator of $\mathbb{E}(\sqrt{n}\hat{%
\phi}_{i,n})^{2}$ in an iid homoscedastic setting. We find a max-t-test with 
$se_{i,n}^{(1)}$ works well in many settings, although never out-performs a
max-test. See \citet[Appendix C]{sm_sig_pred} for high dimensional HAC
estimator theory and a proof that HAC-based $\mathcal{W}_{i,n}$ satisfies $%
\max_{1\leq i\leq p}|\mathcal{W}_{i,n}-\mathcal{W}_{i}|$ $=$ $o_{p}(1/(\ln
(p))^{1+\kappa })$.
\end{remark}

\section{Bootstrap Theory\label{sec:boot_theory}}

We show the dependent wild bootstrap p-value approximation $\hat{p}_{n,%
\mathcal{M}}$ based on $\mathcal{\tilde{T}}_{n}$ $:=$ $\max_{1\leq i\leq p}|%
\sqrt{n}\mathcal{W}_{i,n}\left[ 0,1\right] \widehat{\mathcal{H}}_{i,n}^{-1}%
\widehat{\mathcal{\tilde{G}}}_{i,n}|$\ promotes a valid, consistent test.
Similar arguments extend to the parametric bootstrap $\hat{p}_{n,\mathcal{M}%
}^{(p)}$, and as before extend to bootstrapped p-values based on the average
statistic. We require a bound on block size $b_{n}$ and cross-unit
dependence between $\mathbf{x}_{i,t}v_{t}$ and $\mathbf{x}_{j,t+l}v_{t+l}$,
where $v_{t}$ $:=$ $y_{t}$ $-$ $\mathbb{E}y_{t}$ are innovations from model (%
\ref{mr_model}) under $H_{0}$. Define the centered least squares score
variable%
\begin{equation*}
w_{i,t}:=\mathbf{x}_{i,t}v_{t}-\mathbb{E}\mathbf{x}_{i,t}v_{t}.
\end{equation*}%
Notice $\mathbb{E}\mathbf{x}_{i,t}v_{t}$ $=$ $0$ for all $i$ \textit{if and
only if} $H_{0}$ is true \citep[see][Theorem 2.1]{Hill_2025_maxtest}.

\begin{assumption}
\label{assum_mr_boot} $\ \ \medskip $\newline
$a.$ \textbf{Block Size:} $b_{n}$ $\in $ $\{1,2,...,n-1\}$, $b_{n}$ $%
\rightarrow $ $\infty $ and $b_{n}$ $=$ $o(n^{1/3})$.$\medskip $\newline
$b.$ \textbf{Cross-Unit Dependence:} $\max_{n,t}|\mathbb{E}w_{i,t}w_{j,t+l}|$
$\leq $ $\mathcal{C}(i,j)(l\vee 1)^{-2-\varpi }$ for each $l$ $\in $ $%
\mathbb{N}$ and some $\varpi $ $>$ $0$, where $\max_{i,j}\{\mathcal{C}%
(i,j)\} $ $\leq $ $K$ $<$ $\infty $.$\medskip $\newline
$c.$ \textbf{Cross-Unit Long-Run Dependence:} $\max_{1\leq i,j\leq
p}|1/n\sum_{s,t=1}^{n}$$\mathbb{E}w_{i,s}w_{j,t}$ $-$ $\mathfrak{E}_{n}|$ $=$
$O(1/n^{\varrho })$ for some $\varrho $ $>$ $0$ where $\mathfrak{E}_{n}$ $:=$
$\max_{i,j,n}|1/n\sum_{s,t=1}^{n}\mathbb{E}w_{i,s}w_{j,t}|$.
\end{assumption}

\begin{remark}
\normalfont Lemma \ref{lm:mr_boot_gauss} delivers a multiplier bootstrap
Gaussian approximation provided $\ln (p)$ $=$ $o(\sqrt{n}/b_{n}^{3/2})$.
Thus block size $b_{n}$ $=$ $o(n^{1/3})$ is required to ensure $p$ $%
\rightarrow $ $\infty $. The greater is $b_{n}$, the fewer the number of
blocks, hence a dampened upper bound on $p$ .
\end{remark}

\begin{remark}
\normalfont The cross-unit covariance bound $\max_{n,t}|\mathbb{E}%
w_{i,t}w_{j,t+l}|$ $\leq $ $\mathcal{C}(i,j)(l\vee 1)^{-2-\varpi }$ with $%
b_{n}$ $\rightarrow $ $\infty $ supports a Gaussian-to-Gaussian comparison
used to prove Lemma \ref{lm:mr_boot_gauss}, cf. 
\citet[Lemma
3.1]{Chernozhukov_etal2013} and \citet[Theorem 2]{Chernozhukov_etal2015}. It
is used to yield uniform block-wise and asymptotic covariance equivalence 
\begin{equation*}
\mathfrak{D}_{n}:=\max_{1\leq i,j\leq p}\left\vert \frac{1}{n}%
\sum_{t_{1},t_{2}=1}^{n}\mathbb{E}w_{i,t_{1}}w_{j,t_{2}}-\frac{1}{n}%
\sum_{s=1}^{\mathcal{N}_{n}}\sum_{t_{1},t_{2}=(s-1)b_{n}+1}^{sb_{n}}\mathbb{E%
}w_{i,t_{1}}w_{j,t_{2}}\right\vert \rightarrow 0.
\end{equation*}%
$\mathfrak{D}_{n}$ $=$ $0$ under independence. In a time series setting,
therefore, ($b$) restricts cross-unit heterogeneity and dependence while
Assumption \ref{assum:marg-reg} restricts unit-wise serial dependence.
\end{remark}

\begin{remark}
\normalfont($c$) is unavoidable for a Gaussian-to-Gaussian comparison
without a refinement, e.g., on ($b$). In the proof of Lemma \ref%
{lm:mr_boot_gauss} we show $\mathfrak{E}_{n}$ $:=$ $\max_{i,j,n}|1/n%
\sum_{s,t=1}^{n}\mathbb{E}w_{i,s}w_{j,t}|$ exists under ($b$), and that the
identity $\mathbb{E}w_{i,t}w_{j,t+l}$ $=$ $\mathcal{C}(i,j)_{\pm
}|s-t|^{-2-\omega }$ for some $\mathcal{C}(i,j)_{\pm }$ $\in $ $\mathbb{R}$,
a mild strengthening of ($b$), suffices for ($c$) with $\varrho $ $=$ $1$.
\end{remark}

\subsection{Bootstrap Validity}

The proof of HD bootstrap validity follows a standard path: we prove a
first-order bootstrap approximation, a bootstrap Gaussian approximation, and
then validity. Consult \citet[Appendix F]{sm_sig_pred} for proofs. Define%
\begin{equation*}
\mathcal{H}_{i}:=\mathbb{E}\mathbf{x}_{i,t}\mathbf{x}_{i,t}^{\prime }\text{
and }\mathcal{\tilde{G}}_{i,n}:=\frac{1}{n}\sum_{t=1}^{n}\eta _{t}\left\{ 
\mathbf{x}_{i,t}v_{t}-\mathbb{E}\mathbf{x}_{i,t}v_{t}\right\} \text{ with }%
v_{t}:=y_{t}-\mathbb{E}y_{t}.
\end{equation*}

\begin{lemma}
\label{lm:mr_boot_expand}Let Assumptions \ref{assum:marg-reg} and \ref%
{assum_mr_boot} hold with $r$ $\geq $ $4$ and size $\lambda $ $\geq $ $1$.
For any $\{p\}$ satisfying $\ln (p)$ $=$ $o(n^{1/8}b_{n}^{1/8}/\sqrt{\ln (n)}%
)$,%
\begin{equation*}
\left\vert \max_{1\leq i\leq p}\left\vert \sqrt{n}\left[ 0,1\right] \widehat{%
\mathcal{H}}_{i,n}^{-1}\widehat{\mathcal{\tilde{G}}}_{i,n}\right\vert
-\max_{1\leq i\leq p}\left\vert \sqrt{n}\left[ 0,1\right] \mathcal{H}%
_{i}^{-1}\mathcal{\tilde{G}}_{i,n}\right\vert \right\vert =O_{p}\left( \frac{%
\left( \ln (p)\right) ^{2}}{n^{1/4}b_{n}^{1/4}}\right) =o_{p}\left( 1\right)
.
\end{equation*}
\end{lemma}

\begin{remark}
\normalfont If independence were known then block size $b_{n}$ $=$ $1$ and
an intrinsically different bounding argument would be used to yield $\ln (p)$
$=$ $o(n^{1/4})$ as in Lemma \ref{lm:mr_ng_phys}. The trade-off naturally is
a smaller bound on $p$ that depends on block size $b_{n}$. The bound holds
when block size $b_{n}$ $\simeq $ $n^{\zeta }$ for $\zeta $ $\in $ $(0,1)$,
and $\ln (p)=O(n^{c})$ with $c$ $\in $ $(0,(1$ $+$ $\zeta )/8)$: $p$ can be
monotonically larger as the block size is larger, with $\ln (p)$ $\nearrow $ 
$o(n^{1/4}/\sqrt{\ln (n)})$ as $b_{n}$ $\nearrow $ $n$. Since $b_{n}$ $=$ $%
o(n^{1/3})$ under Assumption \ref{assum_mr_boot}.a, clearly $\ln (p)$ $%
\nearrow $ $o(n^{1/6}/\sqrt{\ln (n)})$ as $b_{n}$ $\nearrow $ $n^{1/3}$.
\end{remark}

Now define the sample $\mathfrak{S}_{n}$ $:=$ $\{x_{t},y_{t}\}_{t=1}^{n}$.
Recall the Lemma \ref{lm:mr_gauss_phys} Gaussian process $\{\mathcal{Z}%
_{n}(i)$ $:$ $1$ $\leq $ $i$ $\leq $ $p\}$ with $\mathcal{Z}_{n}(i)$ $\sim $ 
$\mathcal{N}(0,\sigma _{n}^{2}(i))$. Let $\{\mathcal{\tilde{Z}}_{n}(i)$ $:$ $%
1$ $\leq $ $i$ $\leq $ $p\}$ be an independent copy of $\{\mathcal{Z}_{n}(i)$
$:$ $1$ $\leq $ $i$ $\leq $ $p\}$, independent of the asymptotic draw $%
\{x_{t},y_{t}\}_{t=1}^{\infty }$.

We have the following conditional multiplier bootstrap Gaussian
approximation.

\begin{lemma}
\label{lm:mr_boot_gauss}Let Assumptions \ref{assum:marg-reg} and \ref%
{assum_mr_boot} hold with $r$ $\geq $ $8$ and size $\lambda $ $\geq $ $2$.
Let $\{p\}$ be any integer sequence satisfying $\ln (p)$ $=$ $o(\sqrt{n}%
/b_{n}^{3/2})$, let $b_{n}$ $=$ $o(n^{1/3})$. Then 
\begin{equation*}
\rho _{n}^{\ast }:=\sup_{c\geq 0}\left\vert \mathbb{P}\left( \max_{1\leq
i\leq p}\left\vert \sqrt{n}\left[ 0,1\right] \mathcal{H}_{i}^{-1}\mathcal{%
\tilde{G}}_{i,n}\right\vert >c\mid \mathfrak{S}_{n}\right) -\mathbb{P}\left(
\max_{1\leq i\leq p}\left\vert \mathcal{\tilde{Z}}_{n}(i)\right\vert
>c\right) \right\vert \overset{p}{\rightarrow }0.
\end{equation*}
\end{lemma}

\begin{remark}
\normalfont A Gaussian-to-Gaussian comparison requires convergence of
block-wise and long run covariances between $(\mathbf{x}_{i,s}v_{s},\mathbf{x%
}_{j,t}v_{t})$, thus we need at least an $8^{th}$ moment.
\end{remark}

\begin{remark}
\normalfont Contrary to Lemma \ref{lm:mr_boot_expand}, a Gaussian
approximation yields an inverse relationship between $p$ and block size $%
b_{n}$, with $\ln (p)$ $=$ $o(\sqrt{n})$ as $b_{n}$ $\searrow $ $1$, the iid
case.
\end{remark}

\begin{lemma}
\label{lm:mr_boot_Tn}Let Assumptions \ref{assum:marg-reg} and \ref%
{assum_mr_boot} hold with $r$ $\geq $ $8$ and size $\lambda $ $\geq $ $2$.
Assume $\{\mathcal{W}_{i,n}\}$ satisfy (\ref{WnW}), and let $\ln (p)$ $=$ $%
o(n^{\frac{1}{4(1+\kappa )}}$ $\wedge $ $\sqrt{n}/b_{n}^{3/2})$. Then%
\begin{equation*}
\rho _{n}^{\ast \prime }:=\sup_{c\geq 0}\left\vert \mathbb{P}\left( \mathcal{%
\tilde{T}}_{n}>c\mid \mathfrak{S}_{n}\right) -\mathbb{P}\left( \max_{1\leq
i\leq p}\left\vert \mathcal{\tilde{Z}}_{n}(i)\right\vert >c\right)
\right\vert \overset{p}{\rightarrow }0.
\end{equation*}
\end{lemma}

Now the main bootstrap result. The bound on $p$ is complicated in view of
Theorem \ref{thm:mr_gauss_phys} (for $\hat{\phi}_{i,n}$) and Lemmas \ref%
{lm:mr_boot_expand} and \ref{lm:mr_boot_gauss} (for $\sqrt{n}\left[ 0,1%
\right] \widehat{\mathcal{H}}_{i,n}^{-1}\widehat{\mathcal{\tilde{G}}}_{i,n}$%
).

\begin{theorem}
\label{thm:mr:boot}Let Assumptions \ref{assum:marg-reg} and \ref%
{assum_mr_boot} hold with $r$ $\geq $ $8$ and size $\lambda $ $\geq $ $2$,
and let $\{\mathcal{M}_{n}\}$ be any sequence of positive integers, $%
\mathcal{M}_{n}$ $\rightarrow $ $\infty $. Let block size $b_{n}$ $\simeq $ $%
n^{b}$ with $b$ $\in $ $(0,1/3)$. Let $\{p\}$ be any sequence of positive
integers satisfying $\ln (p)$ $=$ $o(n^{s(\mu ,\kappa ,\lambda )}\wedge n^{%
\frac{1}{4(1+\kappa )}}\wedge n^{\frac{1-3b}{2}})$ with $s(\mu ,\kappa
,\lambda )$ characterized in (\ref{s(b,lam)}). Let weights $\{\mathcal{W}%
_{i,n}\}$ satisfy (\ref{WnW}). Then we have $a.$ $\mathbb{P}(\hat{p}_{n,%
\mathcal{M}_{n}}$ $<$ $\alpha )$ $\rightarrow $ $\alpha $ under $H_{0}$; and 
$b.$ $\mathbb{P}(\hat{p}_{n,\mathcal{M}_{n}}$ $<$ $\alpha )$ $\rightarrow $ $%
1$ under $H_{1}$.
\end{theorem}

\begin{remark}
\normalfont A smaller block size $b\downarrow $ yields a greater upper bound
on $p$. From (\ref{s(b,lam)}) it follows $\ln (p)$ $=$ $o(n^{s(\mu ,\kappa
,\lambda )}\wedge n^{\frac{1-3b}{2}})$ when moment growth parameter $\mu $ $%
\in $ $(0,1]$.
\end{remark}

\subsection{Local Power}

Now consider a sequence of local alternatives adjusted to account for a high
dimension, 
\begin{equation*}
H_{1}^{L}:\phi =c\left\{ \ln (p)\right\} ^{2}/\sqrt{n}\text{, where }c=\left[
c_{i}\right] _{i=1}^{p}\in \mathbb{R}^{p}.
\end{equation*}%
Let $\bar{c}$ $:=$ $\lim \inf_{n\rightarrow \infty }\max_{i}|c_{i}|$ $\in $ $%
[0,\infty ]$. The null hypothesis holds when $\bar{c}$ $=$ $0$. The proposed
bootstrap test yields non-trivial local power, with monotonic local power
improvement as $\max_{i}|\phi _{i}|$ deviates from $0$ (as $\bar{c}$ $%
\nearrow $ $\infty $). Notice the conditions of Theorem \ref{thm:mr:boot}
ensure $\ln (p)$ $=$ $o(n^{1/4})$, hence $H_{1}^{L}$\ indeed represents
local-to-null drift.

\begin{theorem}
\label{thm:mr:boot_local}Under the conditions of Theorem \ref{thm:mr:boot}
and $H_{1}^{L}$, $\mathbb{P}(\hat{p}_{n,\mathcal{M}_{n}}$ $<$ $\alpha )$ $%
\rightarrow $ $\alpha $ when $\bar{c}$ $=$ $0$, $\lim \inf_{n\rightarrow
\infty }\mathbb{P}(\hat{p}_{n,\mathcal{M}_{n}}$ $<$ $\alpha )$ $>$ $\alpha $
when $\bar{c}$ $>$ $0$, and $\mathbb{P}(\hat{p}_{n,\mathcal{M}_{n}}$ $<$ $%
\alpha )$ $\nearrow $ $1$ as $\bar{c}$ $\nearrow $ $\infty $.
\end{theorem}

\section{Simulation Study\label{sec:sim}}

We now study the small sample properties of the proposed test. \cite%
{McKeague_Qian_2015} compare their \textit{adaptive resampling test} [ART]
with the Likelihood Ratio test, multiple testing with Bonferroni correction,
centered percentile bootstrap, and higher criticism. They find that in an
iid linear regression setting ART has competitive empirical power and
accurate size overall, and generally dominates all other methods. We
therefore compare our test with ART based on a nonparametric bootstrap and
parametric wild bootstrap.

\subsection{Data Generating Process}

Our simulation design is similar to McKeague and Qian's (%
\citeyear{McKeague_Qian_2015}), with the major differences that $%
\{y_{t},x_{t}\}$ may be serially dependent, $y_{t}$ may be conditionally
heteroscedastic, and $p/n\rightarrow \infty $. Throughout $\epsilon _{t}$ is
iid $\mathcal{N}(0,1)$, and $v_{t}$ is either iid $\mathcal{N}(0,1)$ or
GARCH: 
\begin{equation*}
\begin{tabular}{ll}
E1: & $v_{t}$ $=$ $\epsilon _{t}$ \\ 
E2: & $v_{t}$ $=$ $\sigma _{t}\epsilon _{t}$ with iid $\epsilon _{t}$ $\sim $
$\mathcal{N}(0,1)$ \\ 
& GARCH(1,1) $\sigma _{t}^{2}$ $=$ $1$ $+$ $.3v_{t-1}^{2}$ $+$ $.5\sigma
_{t-1}^{2}$ with initial condition $\sigma _{1}^{2}$ $=$ $1$.%
\end{tabular}%
\end{equation*}%
Covariates are joint normals that are serially independent or dependent.
There are four total cases. Let $\Sigma $ $\in $ $\mathbb{R}^{p\times p}$
satisfy $\Sigma _{i,i}$ $=$ $1$, and $\Sigma _{i,j}$ $=$ $\gamma $ $\forall
i $ $\neq $ $j$ with $\gamma $ $\in $ $\{0,.5,.8\}$. Then 
\begin{equation*}
\begin{tabular}{ll}
C1: & $x_{t}$ $\sim $ $\mathcal{N}(0,\Sigma )$ \\ 
C2: & $x_{t}$ $=$ $Aw_{t}$ $+$ $v_{x,t}$ where $v_{x,t}\sim \mathcal{N}%
(0,I_{p}),$ \\ 
& $A\in $ $\mathbb{R}^{k}$ from a uniform distribution on $[-1,1]$, \\ 
& $w_{i,t}$ $=$ $.5w_{i,t-1}$ $+$ $\epsilon _{i,t}$ where iid $\epsilon
_{i,t}$ $\sim $ $\mathcal{N}(0,1),$ $i=1,...,p.$%
\end{tabular}%
\end{equation*}%
Under C1, $x_{i,t}$ are serially independent, and mutually independent 
\textit{if and only if} $\gamma $ $=$ $0$, and otherwise have
cross-correlation $\gamma $ (under standardization). Under C2, $x_{i,t}$ are
serially and mutually dependent, where $A$ is the same uniform draw for
every sample.\footnote{%
If $1/n\sum_{t=1}^{n}x_{t}x_{t}^{\prime }$ is not positive definite then $A$
is redrawn again. This never took place in our reported study.} The total
covariate vector $\tilde{x}_{t}$ used for the tests is comprised of a lagged 
$y_{t}$ and $x_{t}$: $\tilde{x}_{t}$ $=$ $[y_{t-1},x_{t}^{\prime }]^{\prime
} $.

We use $10$ total models linking $(\tilde{x}_{t},y_{t})$, based on $5$ base
structures: 
\begin{equation*}
\begin{tabular}{ll}
\textbf{Model $\boldsymbol{i}$ (\textit{{\small {null}}})} & $y_{t}=v_{t}$
\\ 
\textbf{Model $\boldsymbol{ii}$ (\textit{{\small {weak/strong sparse}}})} & $%
y_{t}$ $=$ $\phi _{1}x_{1,t}$ $+$ $v_{t}$ with $\phi _{1}$ $\in $ $%
\{.15,.25\}$ \\ 
\textbf{Model $\boldsymbol{iii}$ (\textit{{\small {moderate}}})} & $y_{t}$ $%
= $ $\sum_{i=1}^{p}\phi _{i}x_{i,t}$ $+$ $\varphi y_{t-1}+$ $v_{t}$ where $%
\{\phi _{i}\}_{i=1}^{5}$ $=$ $.15$, \\ 
& $\{\phi _{i}\}_{i=6}^{10}$ $=$ $-.1$, $\{\phi _{i}\}_{i=11}^{p}$ $=$ $0$,
plus $\varphi $ $\in $ $\{.00,.15,.50\}$; \\ 
\textbf{Model $\boldsymbol{iv}$ (\textit{{\small {AR weak/strong sparse}}})}
& $y_{t}$ $=\varphi y_{t-1}$ $+$ $v_{t}$ with $\varphi $ $\in $ $%
\{.25,.50\}; $ \\ 
\textbf{Model $\boldsymbol{v}$ (\textit{{\small {weak dense}}})} & $y_{t}$ $%
= $ $\sum_{i=1}^{p}\phi _{i}x_{i,t}$ $+$ $v_{t}$, \\ 
& $\phi _{i}=\varphi \mathcal{I}\left( i\leq \lfloor p/3\rfloor \right)
-(\varphi /3)\mathcal{I}\left( i\leq \lfloor 2p/3\rfloor \right) $, \\ 
& and $\varphi $ $\in $ $\{.10,.15\}.$%
\end{tabular}%
\end{equation*}%
\cite{McKeague_Qian_2015} use iid $\{v_{t},x_{t}\}$ thus $v_{t}$ under E1
and $x_{t}$\ under C1. They work with Models $i$, $ii$ with $\phi _{1}$ $=$ $%
.25$, and $iii$ with $\varphi $ $=$ $0$ (no AR component). We also
incorporate GARCH shocks $v_{t}$, and AR feedback in both $x_{t}$ and $y_{t}$%
. Under Model $iii$ there may be AR feedback in $x_{t}$ and $y_{t}$, and
Model $iv$ mirrors the sparse signal $ii$ in an AR setting. Models $ii$ with 
$\phi _{1}$ $=$ $.15$ and $iv$ with $\varphi $ $=$ $.25$ are sparse with a
weak signal because the test uses \textit{all} covariates $\tilde{x}_{t}$ $=$
$[y_{t-1},x_{t}^{\prime }]^{\prime }$, but only $x_{1,t}$ or $y_{t-1}$
matters. Thus Models $iii$ with AR and GARCH feedback, and $iv$, are the
greatest departures from an iid setting.

Model $v$ is a variant of a model in Zhang and Laber's (%
\citeyear{ZhangLaber2015}) numerical study of ART based on a parametric
simulation-based bootstrap. We set $\phi _{i}$ via $\varphi $ small enough
to represent a weak dense alternative. As we will see below, \cite%
{ZhangLaber2015} correctly note that an average statistic may be more adept
at detecting (weak) dense deviations from the null, compared to a
max-statistic.

Model $i$ represents $H_{0}$, while $H_{1}$ holds under Models $ii$-$v$ with
varying degrees of departure from $H_{0}$ (weak or strong, sparse, moderate
or dense signals). Under Models $ii$ and $iv$\ there is only one relevant
predictor that may be a very weak signal. We will see that this case is
especially challenging under GARCH shocks, where generally covariate
dependence improves a detectable signal. In all cases $\{y_{t},x_{t}\}$ are
geometrically $\mathcal{L}_{r}$-\emph{pd} for any $r$ $>$ $1$, hence size $%
\lambda $ $>$ $0$ is arbitrarily large.

The use of GARCH $v_{t}$ under E1 deviates from the Assumption \ref%
{assum:marg-reg}.a condition that $y_{t}$ has all moments finite. Indeed, $%
\mathbb{P}(|v_{t}|$ $>$ $x)$ $\approx $ $cx^{-\kappa }$ for $\kappa $ $\in $ 
$(4,8)$, below the $\mathcal{L}_{8}$-\emph{pd} requirement, cf. Theorem \ref%
{thm:mr:boot}.\footnote{%
See \citet[Theorem 3.1, Corollary 3.5]{Basrak_etal2002} for $x^{\kappa }%
\mathbb{P}(|v_{t}|$ $>$ $x)$ $\rightarrow $ $c$ $\in $ $(0,\infty )$ for
some $\kappa $ $>$ $0$ in the GARCH(1,1) case. Moreover, $\mathbb{E}%
(.3\epsilon _{t}^{2}$ $+$ $.5)^{2}$ $<$ $1$ and $\mathbb{E}(.3\epsilon
_{t}^{2}$ $+$ $.5)^{4}$ $>$ $1$ are easily verified, hence $\kappa $ $\in $ $%
(4,8)$, cf. \cite{BougerolPicard1992}.} This deviation, however, does not
appear to matter in small samples.\footnote{%
In an experiment not reported here we used GARCH $v_{t}$ with iid $\epsilon
_{t}$ drawn from a standardized \textit{truncated normal} with support $%
[-4,4]$. In this case Assumption \ref{assum:marg-reg}.a holds. Test results
are essentially identical to the standard GARCH case treated here: the
dominating force is the random volatility itself and not the lack of higher
moments \textit{per se}.}

We simulate $n$ $+$ $500$ observations and retain the final $n$ observations
for analysis, with sample sizes $n$ $\in $ $\{100,200,400\}$. Block size for
the dependent wild bootstrap is $b_{n}$ $=$ $5[n^{b}]$ with $b$ $=$ $1/6$ so
that $b_{n}$ $\in $ $\{10,10,15\}$. Hence from Theorem \ref{thm:mr:boot}\ we
need $\ln (p)$ $=$ $o(n^{\{1/2-3b/2\}\wedge s(\mu ,\kappa ,\lambda )})$ $=$ $%
o(n^{1/4})$: notice $s(\mu ,\kappa ,\lambda )$ $=$ $1/4$ in (\ref{s(b,lam)})
under Gaussianicity and arbitrary $\lambda $ (due to geometric memory). We
therefore use $p$ $\in $ $\{10,50,100,175,\bar{p}\}$ with $\bar{p}$ $=$ $%
\lfloor 20\exp \{n^{1/4}\}/\sqrt{\ln (n)}\rfloor $ $\in $ $\{220,373,715\}$.

\subsection{Tests}

The three tests performed are as follows. Note $\{x_{i,t},y_{t}\}$ are
standardized for all tests.

\subsubsection{Max- and Ave-Tests}

We compute the max-statistic $\mathcal{\hat{T}}_{n}$ $:=$ $\max_{1\leq i\leq
p}|\sqrt{n}\mathcal{W}_{n,i}\hat{\phi}_{i,n}|$ with $\hat{\phi}_{i,n}$ $=$ $%
\sum_{t=1}^{n}(x_{i,t}$ $-$ $\bar{x}_{i,n})(y_{t}$ $-$ $\bar{y}%
_{n})/\sum_{t=1}^{n}(x_{i,t}$ $-$ $\bar{x}_{i,n})^{2}$ and parametric wild
bootstrap p-value $\hat{p}_{n,\mathcal{M}}^{(p)}$ with $\mathcal{M}_{n}$ $=$ 
$1000$ bootstrap samples. The parametric wild bootstrap out-performed the
dependent wild bootstrap in experiments not reported here: generally the
latter leads to over-sized tests. The bootstrap procedure uses iid draws $%
\{\xi _{l}\}_{l=1}^{[n/b_{n}]}$ from $\mathcal{N}(0,1)$.

Weights are $\mathcal{W}_{n,i}$ $=$ $1$ or $\mathcal{W}_{n,i}$ $=$ $%
1/se_{i,n}^{(1)}$ with the classic least squares standard error $%
se_{i,n}^{(1)}$. We call these the max-test and max-t-test. The use of a HAC
estimator $se_{i,n}^{(2)}$ to yield a true max-t-test does not improve
results irrespective of the chosen bandwidth (hence irrespective of using
automatic bandwidth selection), and generally leads to size distortions and
low power, depending on $n$, error and covariate case. Indeed, a max-t-test
in either form never dominated a max-test in this study under Models $i,$ $%
ii,$ $iii,$ and $v$. A max-t-test with $\mathcal{W}_{n,i}$ $=$ $%
1/se_{i,n}^{(1)}$ yields slightly larger (though nearly equivalent)
empirical power in Model $iv$. Overall this matches evidence in 
\citet[Fn.
13]{HillMotegi2020}, \cite{Hill_2025_maxtest} and \cite{HillLi2025}. We
therefore only report max-test results. We reject a max-test at level $%
\alpha $ when $\hat{p}_{n,\mathcal{M}}^{(p)}$ $<$ $\alpha $.\footnote{%
All computations are performed on the Longleaf cluster at the University of
North Carolina-Chapel Hill, using \textsc{Matlab 2024a} and \textsc{R 4.4.0}%
, with a SLURM scheduler.}

We also study the average $\widehat{\mathfrak{A}}_{n}$ $:=$ $\sum_{i=1}^{p}|%
\sqrt{n}\mathcal{W}_{n,i}\hat{\phi}_{i,n}|$ with weight $\mathcal{W}_{n,i}$ $%
=$ $1$. As above, the use of $\mathcal{W}_{n,i}$ $=$ $1/se_{i,n}^{(1)}$ or $%
1/se_{i,n}^{(2)}$ never yielded a dominant test.

\subsubsection{ART}

We need to compute and rank bootstrapped confidence intervals using a bias
correction detailed below. Compute $\hat{\phi}_{\hat{l}_{n},n}$ with $\hat{l}%
_{n}$ $=$ $\arg \max_{1\leq l\leq p}|\hat{\phi}_{l,n}|$, let $T_{n}$ $:=$ $%
\hat{\phi}_{\hat{l}_{n},n}/se_{\hat{l}_{n},n}$ with the conventional
standard error $se_{\hat{l}_{n},n}$, and let $(\hat{\phi}_{\hat{l}%
_{n},n}^{\ast },se_{\hat{l}_{n},n}^{\ast })$ be bootstrap versions. We
consider two procedures: nonparametric bootstrap [NB] (resampling with
replacement) as in \cite{McKeague_Qian_2015}, which is only suitable for
independent observations, and parametric (dependent) wild bootstrap [PWB].
The ART procedure in the latter case does not involve resampling \textit{per
se}, and is more akin to the weak identification robust wild bootstrap in 
\cite{Hill2021}, in the spirit of standard error computation in \cite%
{Andrews1999} and weak identification robust critical values in \cite%
{Andrews_Cheng_2012}. Although we do not present a supporting theory here,
the ideas in \cite{Hill2021} will extend to dependent data by using blocking
in order to handle the lack of uniformity detailed in \cite%
{McKeague_Qian_2015}. See \citet[Appendix D]{sm_sig_pred} for complete NB
and PWB details. Brief NB details follow below.

The bias correction is derived as follows. Let $T_{n}^{\ast }$ $:=$ $\hat{%
\phi}_{\hat{l}_{n},n}^{\ast }/se_{\hat{l}_{n},n}^{\ast }$. Denote a nuisance
sequence $\{\lambda _{n}(\alpha )\}$, $\lambda _{n}(\alpha )$ $>$ $0$, $%
\lambda _{n}(\alpha )$ $\rightarrow $ $\infty $ and $\lambda _{n}(\alpha )=o(%
\sqrt{n})$ for a level $\alpha $ test. Compute for $1000$ independently
drawn bootstrap samples 
\begin{align*}
& & \mathcal{A}_{n}^{\ast }(\alpha )& :=\sqrt{n}\left( \hat{\phi}_{\hat{l}%
_{n},n}^{\ast }-\hat{\phi}_{\hat{l}_{n},n}\right) \times \mathcal{I}\left(
\left\vert T_{n}^{\ast }\right\vert >\lambda _{n}(\alpha )\text{ }\cup \text{
}\left\vert T_{n}\right\vert >\lambda _{n}(\alpha )\right) \\
& & & \text{ \ \ \ \ \ \ \ \ \ \ \ \ \ \ \ }+\mathbb{V}_{n}^{\ast }\times 
\mathcal{I}\left( \left\vert T_{n}^{\ast }\right\vert \leq \lambda
_{n}(\alpha )\text{ }\cap \text{ }\left\vert T_{n}\right\vert \leq \lambda
_{n}(\alpha )\right) ,
\end{align*}%
where $\mathbb{V}_{n}^{\ast }$ under NB is detailed in 
\citet[eq. (4) and p.
1431]{McKeague_Qian_2015}. Let $[L_{\alpha /2},U_{\alpha /2}]$ be the
empirical lower and upper $\alpha /2$\ quantiles of $\mathcal{A}_{n}^{\ast
}(\alpha )$. We reject the null at level $\alpha $ when $\sqrt{n}\hat{\phi}_{%
\hat{l}_{n},n}$ $\notin $ $[L_{\alpha /2},U_{\alpha /2}]$. A bootstrapped
p-value is also easily computed as $1/1000\sum_{j=1}^{1000}\mathcal{I}(|%
\mathcal{A}_{n,j}^{\ast }(\alpha )|$ $>$ $|\sqrt{n}\hat{\phi}_{\hat{l}%
_{n},n}|)$, with $\mathcal{A}_{n,j}^{\ast }(\alpha )$ the $j^{th}$ bootstrap
sample $\mathcal{A}_{n}^{\ast }(\alpha )$. Overall the PWB never dominated
the NB under any hypothesis, yielding an under-sized test. We therefore only
report NB results.

The tuning parameter $\lambda _{n}(\alpha )$ is set as follows. At nominal
size $\alpha $, and using a Bonferroni inequality, put $\lambda _{n}$ $=$ $%
\lambda _{n}(\omega ,\alpha )$ $:=$ $\max \{\sqrt{\omega \ln (n)},\mathcal{Z}%
_{\alpha /(2p)}\}$ where $\mathcal{Z}_{\alpha /(2p)}$ is the upper $\alpha
/(2p)$ quantile of $\mathcal{N}(0,1)$. The scale $\omega $ is selected by a
separate parametric bootstrap with $1000$ draws yielding $\mathbf{\hat{\phi}}%
_{\hat{l}_{n},n}^{\ast }$. Define $\mathcal{R}_{n}$ $:=$ $\sqrt{n}|\mathbf{%
\hat{\phi}}_{\hat{l}_{n},n}^{\ast }$ $-$ $\hat{\phi}_{n}|$, let $\mathcal{R}%
_{[j]}$ be the ranked values $\mathcal{R}_{[1]}$ $\geq $ $\mathcal{R}_{[2]}$ 
$\geq $ $\cdots $, and pick $\omega ^{\ast }$ such that $\lambda _{n}(\omega
^{\ast },\alpha )$ $=$ $\mathcal{R}_{[\alpha n]}$. This yields the double
bootstrap as detailed in \cite{McKeague_Qian_2015}.

\subsection{Results}

Rejection frequencies at the $5\%$ level are shown in Figures H.1-H.10 in
the supplemental material \citet[Appendix
H]{sm_sig_pred}. Each test (max, ave, ART) works monotonically better as
cross-covariate dependence increases, through mutual (in)dependence with $%
\gamma $ $\in $ $\{0,.5,.8\}$ and AR feedback. Indeed, under AR covariates
C2 the tests perform similarly, with mildly low size under $H_{0}$, and
empirical power at or near $100\%$ under each $H_{1}$. We therefore only
report results under covariate case C1, cases $\gamma $ $\in $ $\{0,.8\}$.

\subsubsection{Ave-test}

The ave-test shows size distortions in several cases: over-sized with strong
mutual dependence across covariates $\gamma $ $=$ $.8$ at $n$ $=$ $100$, and
well under-sized as $p$ increases under mutually independent covariates.
Covariate dependence extends to estimator dependence, which improves the
ave-test's ability to detect deviations, especially for moderate to large
values of $p$. In most deviations from the null the ave-test is (potentially
massively) trumped by both the max-test and ART, while the ave-test
especially struggles with mutually independent covariates and/or large
dimension $p$.

The ave-test is especially vulnerable to high dimension. In multiple cases
power falls heavily as $p$ increases compared to the max-test and ART.
Ultimately the reason is a max-statistic and ART work parsimoniously with
the single most informative (weighted) regression slope, while the
ave-statistic smooths over many possibly redundant or statistically
uninformative estimators (e.g. Models $ii,$ $iii,$ $iv$). The latter
significantly augments noise in the test, appearing as low empirical power
under high dimension. For example, in Model $iii$ with mutually independent
covariates ($\gamma $ $=$ $0$) and $n$ $\in $ $\{100,200\}$, empirical power
is roughly $5\%$ for large $p$, yielding zero net power. See, e.g., Figure
H.6, first and third column.

The only case where the ave-test may have an advantage is under Model $v$
with a weak dense deviation from the null (in certain error and covariate
cases), as intuited by \cite{ZhangLaber2015}. For example, under a very weak
signal $\varphi $ $=$ $.10$ with $n$ $\in $ $\{200,400\}$ and covariate
mutual independence $\gamma $ $=$ $0$, empirical power rises faster and over
higher dimensions. The result is more attenuated with a stronger signal $%
\varphi $ $=$ $.15$ and greater noise caused by GARCH innovations (see
Figure H.8, third row, third column). Overall, however, the ave-test is
either roughly equivalent to or dominated by the max-test and ART.

\subsubsection{Max-test and ART}

We finish this section by focusing on the max-test and ART. In general these
tests perform better with covariate dependence and limited error volatility
(i.e. iid shocks), where neither is surprising. Empirical power for all
tests plunge with GARCH volatility, and increases in $n$. In many cases the
ART and max-test are quite similar in their responses to covariate
dependence, error properties, and degree of signal or deviation from the
null, in particular at $n$ $=$ $400$.

A few notable differences, however, do exist. Overall the max-test is best
suited for detecting subtle deviations from the null (weak signal), matching
evidence elsewhere \citep[e.g.][]{HillLi2025,Hill_2025_maxtest}. This is
best revealed by Model $ii$ (one relevant covariate) with weak signal $\phi
_{1}$ $=$ $.15$, and Model $iv$ (one relevant AR covariate) with $\varphi $ $%
=$ $.25$. See Figures H.2, H.3, H.7 and H.8. In Model $ii$\ with mutually
independent covariates the max-test dominates, and is dominant even under
mutual dependence ($\gamma $ $=$ $.8$) when the signal is very weak ($\phi
_{1}$ $=$ $.15$), or with GARCH errors and independent covariates when the
signal is stronger ($\phi _{1}$ $=$ $.25$). In Model $iv$ with a weak signal 
$\varphi $ $=$ $.25$ the max-test strongly dominates at smaller sample sizes 
$n$ $\in $ $\{100,200\}$, an advantage that all but disappears at $n$ $=$ $%
400$ or with a stronger signal $\varphi $ $=$ $.50$.

Thus, the ART procedure and its reliance on the most informative covariate
index $\hat{\imath}_{n}$ has real merit, as revealed in \cite%
{McKeague_Qian_2015}, subsequent comments in, e.g., \cite{ZhangLaber2015},
and our study. That said, considering that ART is a type of max-test, 
\textit{why} ART works well in some cases and less-so in others must be due
to its use and control of $\hat{\imath}_{n}$\ in a bootstrap setting akin to
ideas in \cite{Andrews_Cheng_2012} and \cite{Hill2021}. That ART and its
bootstrap theory operate on a stochastic index $\hat{\imath}_{n}$ adds
complexity to the procedure and limit theory by introducing weak
identification, even under the simulation refinement in \cite{ZhangLaber2015}%
. In this paper we only operate on the max- or ave-statistic with a wild
bootstrap, never requiring the optimal covariate index. The streamlined
procedure leads to size control, and dominant or competitive power for the
max-test, in particular against very weak and sparse deviations from the
null.

\section{Conclusion\label{sec:conclus}}

We present block bootstrap max- and ave-tests for detecting the presence of
significant predictors in a high dimensional setting. The number of
covariates to be screened may be large $p$ $>>$ $n$, growing at an
exponential rate, and we allow for weakly dependent and heterogeneous
(possibly non-stationary) data. Our theory gives social and material
scientists broad access to cross-sections, time series and panels where high
dimensions are frequently encountered. The work here should extend with
moderate effort to coherence estimation under mutual and serial dependence
in high dimensional settings. Another interesting prospect is to extend the
optimal test class in \cite{AndrewsPloberger1994} to the HD setting here,
where the nuisance term is the covariate index $i$. The goal here will be to
place HD max- and ave-tests properly within a field of optimal tests in the
sense of greatest weighted average power. We leave these ideas for future
work.

\setcounter{equation}{0} \renewcommand{\theequation}{{\thesection}.%
\arabic{equation}} \appendix

\section{Appendix: Technical Proofs\label{sec:app}}

Required supporting results Lemmas B.2-B.4 presented in 
\citet[Appendix
B]{sm_sig_pred} deliver concentration bounds, and max-LLN's for key
components in this paper. In general we assume $\mathbb{E}y_{t}$ $=$ $0$,
drop the constant term $\delta _{i}$ from (\ref{mr_model}), and set $\mathbf{%
x}_{i,t}$ $=$ $x_{i,t}$ to improve clarity at no real cost. The bootstrap
proofs assume the number of blocks $\mathcal{N}_{n}$ $=$ $n/b_{n}$ is an
integer to simplify arguments.\medskip

\noindent \textbf{Proof of Lemma \ref{lm:mr_ng_phys}.} Use the triangle
inequality to yield%
\begin{equation*}
\left\vert \max_{1\leq i\leq p}\left\vert \sqrt{n}\hat{\phi}%
_{i,n}\right\vert -\max_{1\leq i\leq p}\left\vert \frac{1}{\sqrt{n}}%
\sum\nolimits_{i=1}^{n}\mathcal{H}_{i}^{-1}x_{i,t}v_{t}\right\vert
\right\vert \leq \max_{1\leq i\leq p}\left\vert \sqrt{n}\hat{\phi}_{i,n}-%
\frac{1}{\sqrt{n}}\sum\nolimits_{i=1}^{n}\mathcal{H}_{i}^{-1}x_{i,t}v_{t}%
\right\vert .
\end{equation*}%
It suffices to bound the right hand side. Under $H_{0}$ for all $i$, $%
v_{i,t} $ $=$ $v_{t}$ $=$ $y_{t}$ $-$ $\mathbb{E}y_{t}$ is $\mathcal{L}_{2}$%
-orthogonal to all $x_{i,t}$: $\mathbb{E}x_{i,t}v_{t}$ $=$ $0$. Therefore, 
\begin{equation*}
\sqrt{n}\hat{\phi}_{i,n}=\frac{1}{\sqrt{n}}\sum_{i=1}^{n}\mathcal{H}%
_{i}^{-1}x_{i,t}v_{t}+\sqrt{n}\left( \widehat{\mathcal{H}}_{i}^{-1}-\mathcal{%
H}_{i}^{-1}\right) \frac{1}{n}\sum_{i=1}^{n}x_{i,t}v_{t}.
\end{equation*}%
Define 
\begin{equation}
\left\Vert \check{z}\right\Vert _{n,p}:=\max_{i,t}\left\{ \left\Vert
x_{i,t}\right\Vert _{p}\vee \left\Vert y_{t}\right\Vert _{p}\right\} ,
\label{z}
\end{equation}%
and note $\lim \sup_{n\rightarrow \infty }||\check{z}||_{n,p}$ $\leq $ $%
ap^{\mu }$ for finite $a>0$ and $\mu $ $\in $ $[0,\infty )$ under Assumption %
\ref{assum:marg-reg}. Set 
\begin{equation*}
\alpha :=\frac{2}{2\mu +1}\in (0,2].
\end{equation*}%
Then $\limsup_{p\rightarrow \infty }\{p^{1/2-1/\alpha }\limsup_{n\rightarrow
\infty }||\check{z}||_{n,p}\}$ $<$ $\infty $, hence (\ref{PHI_z}) implies $%
\limsup_{r\rightarrow \infty }r^{1/2-1/\alpha }$ $\times $ $%
\limsup_{n\rightarrow \infty }\{\check{\Theta}_{n}^{(r)}\}$ $<$ $\infty $.
Now use Lemma B.4 in \cite{sm_sig_pred}\ for $x_{i,t}^{2}$ and $x_{i,t}v_{t}$
to deduce 
\begin{equation*}
\max_{1\leq i\leq p}\left\vert \widehat{\mathcal{H}}_{i}-\mathcal{H}%
_{i}\right\vert =O_{p}\left( \frac{\ln (p)}{\sqrt{n}}\right) \text{ and }%
\max_{1\leq i\leq p}\left\vert \frac{1}{n}\sum_{i=1}^{n}(x_{i,t}v_{t}-%
\mathbb{E}x_{i,t}v_{t})\right\vert =O_{p}\left( \frac{\ln p)}{\sqrt{n}}%
\right) .
\end{equation*}%
This yields for any $\{p\}$ satisfying $\ln (p)$ $=$ $o(n^{1/4})$,%
\begin{equation*}
\left\vert \max_{1\leq i\leq p}\left\vert \sqrt{n}\hat{\phi}%
_{i,n}\right\vert -\max_{1\leq i\leq p}\left\vert \frac{1}{\sqrt{n}}%
\sum_{i=1}^{n}\mathcal{H}_{i}^{-1}x_{i,t}v_{t}\right\vert \right\vert
=O_{p}\left( \sqrt{n}\frac{\ln (p)}{\sqrt{n}}\frac{\ln (p)}{\sqrt{n}}\right)
=o_{p}(1).\text{ }\blacksquare
\end{equation*}%
\textbf{Proof of Lemma \ref{lm:mr_gauss_phys}.} Define%
\begin{equation*}
\mathcal{\hat{Z}}_{n}(i):=\frac{1}{\sqrt{n}}\sum_{t=1}^{n}z_{i,t}\text{
where }z_{i,t}:=\left[ 0,1\right] \mathcal{H}_{i}^{-1}x_{i,t}v_{t}.
\end{equation*}%
Thus $\sigma _{n}^{2}(i)$ $=$ $\mathbb{E}\mathcal{\hat{Z}}_{n}^{2}(i)$ and $%
\sigma ^{2}(i)$ $=$ $\lim_{n\rightarrow \infty }\sigma _{n}^{2}(i)$. Recall $%
\mathcal{Z}_{n}(i)$ $\sim $ $\mathcal{N}(0,\sigma _{n}^{2}(i))$. Then 
\begin{equation*}
\rho _{n}:=\sup_{c\geq 0}\left\vert \mathbb{P}\left( \left\vert \max_{1\leq
i\leq p}\left\vert \mathcal{\hat{Z}}_{n}(i)\right\vert \right\vert >c\right)
-\mathbb{P}\left( \left\vert \max_{1\leq i\leq p}\left\vert \mathcal{Z}%
_{n}(i)\right\vert \right\vert >c\right) \right\vert .
\end{equation*}

We first prove that $\sigma _{n}^{2}(i)$ is uniformly bounded. Assumption %
\ref{assum:marg-reg}.c implies $\min_{i,n}\sigma _{n}^{2}(i)$ $>$ $0$. The
proof of Lemma B.4, in conjunction with Assumption \ref{assum:marg-reg}.c
and (\ref{PHI_z}), imply $z_{i,t}$ is $\mathcal{L}_{r/2}$-\emph{pd}, $r$ $%
\geq $ $4$, with coefficient accumulation $\max_{1\leq t\leq n}\{\Theta
_{i,t}^{(p/2)}\vee \tilde{\Theta}_{t}^{(p/2)}\}$ $\leq $ $K||\check{z}%
||_{n,p/2}$ with $||\check{z}||_{n,p}$ $:=$ $\max_{i,t}\{||x_{i,t}||_{p}$ $%
\vee $ $||y_{t}||_{p}\}$ as in (\ref{z}). See also arguments below. Now use
Lemma B.2.a with Assumption \ref{assum:marg-reg}.a to deduce 
\begin{equation*}
\max_{i,n}\sigma _{n}^{2}(i)\leq 2\max_{i,n}\max_{1\leq t\leq n}\left\{
\Theta _{i,t}^{(p/2)}\vee \tilde{\Theta}_{t}^{(p/2)}\right\} \leq \max_{n}||%
\check{z}||_{n,1}<\infty .
\end{equation*}

Next we prove $\rho _{n}$ $\rightarrow $ $0$. If we verify Conditions 1 and
3 in \cite{ChangChenWu2024} then by their Theorem 3($i$),($ii$)\ and the
mapping theorem $\rho _{n}$ $=$ $O(g_{n})$ for some sequence $\{g_{n}\}$ to
be characterized in Step 4 below. Condition 3 (non-degeneracy) holds by
Assumption \ref{assum:marg-reg}.c.

Consider Condition 1 (sub-exponential tails). Define 
\begin{equation}
\psi _{\zeta }(x):=\exp \{\left\vert x\right\vert ^{\zeta _{1}}\}-1,\text{ }%
\zeta _{1}>0.  \label{psi(x)}
\end{equation}%
We require constants $\{B_{n}\}$ with $\inf_{n\in \mathbb{N}}B_{n}$ $\geq $ $%
1$ and some finite $\zeta _{1}$ $\geq $ $1$ such that 
\begin{equation*}
\left\Vert z_{i,t}\right\Vert _{\psi _{\zeta _{1}}}:=\inf \left\{ \lambda >0:%
\mathbb{E}\left[ \psi _{\zeta _{1}}\left( z_{i,t}/\lambda \right) \right]
\leq 1\right\} \leq B_{n}\text{ }\forall t\in \{1,...,n\}\text{, }\forall
i\in \{1,...,p\}.
\end{equation*}%
Trivially $\mathbb{E}[\psi _{\zeta _{1}}\left( z_{i,t}/\lambda \right) ]$ $%
\leq $ $1$ implies $\max_{i,n,t}\mathbb{E}[\exp \{|z_{i,t}/\lambda |^{\zeta
_{1}}\}]$ $\leq $ $2$. The latter upper bound $2$ does not play any role in
the proof of \cite{ChangChenWu2024}'s Theorem 3, and can be replaced by any
finite $\varpi $ $>$ $0$. The chosen upper bound $\varpi $ affects $g_{n}$ $%
\rightarrow $ $0$ only by a scalar multiple, which is irrelevant here.%
\footnote{%
See \citet[proof of Theorem 1]{ChangChenWu2024}, see the proofs of their
supporting Lemmas 4-7, and consult their cited sources 
\citet[proof of Lemma
L1]{ChangJiangShao2023},\ \citet[Lemma A.2]{ChangTangWu2013}, and %
\citet[Theorem 6.2]{Rio2017}}

Thus, setting $B_{n}$ $=$ $1$, we need only show for some $\varpi $ $>$ $0$,%
\begin{equation}
\max_{i,n,t}\mathbb{E}\left[ \exp \{\left\vert z_{i,t}\right\vert ^{\zeta
_{1}}\}\right] \leq \varpi .  \label{Elt2}
\end{equation}%
We prove (\ref{Elt2}) in three steps.\medskip

\textbf{Step 1 (}$z_{i,t}$ \textbf{physical dependence). }By (B.11) in %
\citet[proof of Lemma B.4]{sm_sig_pred}: 
\begin{equation*}
\left\Vert z_{i,t}-z_{i,t}^{\prime }(m)\right\Vert _{r/2}\leq K\left\{
\theta _{i,t}^{(r)}(m)\vee \tilde{\theta}_{t}^{(r)}(m)\right\} :=K\check{%
\theta}_{i,t}^{(r)}(m)\text{ for }r\geq 4.
\end{equation*}%
Now, by supposition and Minkowski and Cauchy-Schwartz inequalities,%
\begin{align}
& & & \check{\theta}_{i,t}^{(r/2)}(m)\leq Kd_{i,t}^{(r/2)}\left\{ \mathcal{I}%
\left( m=0\right) +m^{-\lambda -\iota }\mathcal{I}\left( m>0\right) \right\} 
\text{ for some }r\geq 4\text{ and }\lambda >2  \label{theta_D} \\
& & & d_{i,t}^{(r/2)}\leq 2\left\Vert z_{i,t}\right\Vert _{r/2}\leq 2\left\{
\left\Vert x_{i,t}\right\Vert _{r}\vee \left\Vert y_{t}\right\Vert
_{r}\right\} \leq 2ar^{\mu }=2^{\mu }ar^{\mu },  \notag
\end{align}%
where the final inequality follows from Assumption \ref{assum:marg-reg}.a($ii
$),b with $(a,\mu )$ $>$ $0$. By construction $z_{i,t}$ is therefore $%
\mathcal{L}_{r/2}$-\textit{pd}, $r$ $\geq $ $4$, with size $\lambda $%
.\medskip 

\textbf{Step 2 (}$z_{i,t}$ \textbf{sub-exponential tails). }Define $\check{%
\theta}_{i,t}^{(r)}(m)$ $:=$ $\theta _{i,t}^{(r)}(m)\vee \tilde{\theta}%
_{t}^{(r)}(m)$ and accumulations $\check{\Theta}_{i,n}^{(r)}(m)$ $:=$ $%
\sum_{j=m}^{\infty }\max_{1\leq t\leq n}\check{\theta}_{i,t}^{(r)}(j)$ and $%
\check{\Theta}_{i}^{(r)}$ $:=$ $\limsup_{n\rightarrow \infty }\check{\Theta}%
_{i,n}^{(r)}(0)$. Under Assumption \ref{assum:marg-reg}.a($ii$), cf. (\ref%
{theta_D}), $\lim \sup_{q\rightarrow \infty }q^{-\mu }\max_{i,n}\check{\Theta%
}_{i}^{(q)}$ $<$ $\infty $. Hence by the argument used to prove Lemma B.2.b,
for some $\mathcal{C},\mathcal{K}$ $\in $ $(0,\infty )$ and $\alpha $ $%
\equiv $ $2/(2\mu $ $+$ $1)$, 
\begin{equation*}
\mathbb{P}\left( \left\vert \frac{1}{\sqrt{n}}\sum_{t=1}^{n}z_{i,t}\right%
\vert >u\right) \leq \mathcal{C}\exp \left\{ -\mathcal{K}u^{\alpha }\right\}
.
\end{equation*}

\textbf{Step 3 ((\ref{Elt2})). }By the arguments used to prove Lemma B.2(b),
and a change of variables yields for some finite $\tilde{c}_{1},\tilde{c}%
_{2} $ $>$ $0$, $\gamma _{1}/2$ $\geq $ $1$ and $\zeta _{1}$ $\in $ $(0,1)$, 
\begin{align*}
& & \max_{i,n,t}\mathbb{E}\left[ \exp \left\{ \left\vert z_{i,t}\right\vert
^{\zeta _{1}}\right\} \right] & =1+\max_{i,n,t}\int_{1}^{\infty }\mathbb{P}%
\left( \left\vert z_{i,t}\right\vert >\left( \ln u\right) ^{1/\zeta
_{1}}\right) du \\
& & & =1+\zeta _{1}\max_{i,n,t}\int_{1}^{\infty }\mathbb{P}\left( \left\vert
z_{i,t}\right\vert >v\right) \exp \left\{ v^{\zeta _{1}}\right\} v^{\zeta
_{1}-1}dv \\
& & & \leq 1+\zeta _{1}\tilde{c}_{1}\int_{1}^{\infty }\exp \{-\tilde{c}%
_{2}v^{\gamma _{1}/2}\}\exp \left\{ v^{\zeta _{1}}\right\} v^{\zeta _{1}-1}dv
\\
& & & \leq 1+\zeta _{1}\tilde{c}_{1}\int_{1}^{\infty }\exp \{v^{\zeta _{1}}-%
\tilde{c}_{2}v^{\gamma _{1}/2}\}dv.
\end{align*}%
It suffices to prove $\int_{1}^{\infty }\exp \left\{ v^{\zeta _{1}}-\tilde{c}%
_{2}v^{\gamma _{1}/2}\right\} dv\leq K$. Notice $\gamma _{1}$ $>$ $2\zeta
_{1}$ $\in $ $(0,2)$ given $\gamma _{1}/2$ $\geq $ $1$ and $\zeta _{1}$ $\in 
$ $(0,1)$. Hence, for some finite \underline{$v$} $>$ $1$ 
\begin{align*}
& & \int_{1}^{\infty }\exp \left\{ v^{\zeta _{1}}-\tilde{c}_{2}v^{\gamma
_{1}/2}\right\} dv& \leq K+\int_{\underline{v}}^{\infty }\exp \left\{ -\frac{%
\tilde{c}_{2}}{2}v^{\gamma _{1}/2}\right\} dv \\
& & & \leq K+\int_{\underline{v}}^{\infty }\exp \left\{ -\frac{\tilde{c}_{2}%
}{2}v\right\} dv=K+\frac{2}{\tilde{c}_{2}}\exp \left\{ -\frac{\tilde{c}_{2}}{%
2}\underline{v}\right\} .
\end{align*}%
This proves (\ref{Elt2}), hence $\rho _{n}$ $=$ $O(g_{n})$.\medskip

\textbf{Step 4 ($g_{n}$).} It remains to characterize sequences $\{g_{n}\}$
such that $g_{n}$ $\rightarrow $ $0$. Define \textit{dependence adjusted
norms} \citep[cf.][]{WuWu2016}: for $\alpha ,v$ $\in $ $(0,\infty )$, 
\begin{equation*}
\left\Vert z_{i,\cdot }\right\Vert _{n,r,\alpha }:=\sup_{m\geq
1}(m+1)^{\alpha }\check{\Theta}_{i,n}^{(r)}(m)\text{ and }\left\Vert
z_{i,\cdot }\right\Vert _{n,\psi _{\nu },\alpha }:=\sup_{r\geq
2}r^{-v}\left\Vert z_{i,\cdot }\right\Vert _{n,r,\alpha }
\end{equation*}%
and $\Psi _{n,r,\alpha }$ $:=$ $\max_{1\leq i\leq p}||z_{i,\cdot
}||_{n,r,\alpha }$ and $\Phi _{n,\psi _{v},\alpha }$ $:=$ $\max_{1\leq i\leq
p}||z_{i,\cdot }||_{n,\psi _{\nu },\alpha }$. From (\ref{theta_D}),%
\begin{align*}
& & \check{\Theta}_{i,n}^{(r)}(m)& \leq 2\max_{1\leq t\leq n}\left\{
\left\Vert x_{i,t}\right\Vert _{2r}\vee \left\Vert y_{t}\right\Vert
_{2r}\right\} \sum_{j=m}^{\infty }\left[ \mathcal{I}\left( j=0\right)
+j^{-\lambda -\iota }\mathcal{I}\left( j>0\right) \right]  \\
& & & \leq K\max_{1\leq t\leq n}\left\{ \left\Vert x_{i,t}\right\Vert
_{2r}\vee \left\Vert y_{t}\right\Vert _{2r}\right\} \times \left[ \mathcal{I}%
\left( m=0\right) +m^{-(\lambda +\iota )/2}\mathcal{I}\left( m>0\right) %
\right]  \\
& & & \leq Kr^{\mu }\times \left[ \mathcal{I}\left( m=0\right) +m^{-(\lambda
+\iota )/2}\mathcal{I}\left( m>0\right) \right] .
\end{align*}%
Hence with $\alpha $ $=$ $\lambda /2$ $>$ $1$ and $\nu $ $=$ $\mu $, for
some $\mathcal{K}_{\lambda }$ $>$ $0$ that depends on $\lambda $, 
\begin{align*}
& & & \Psi _{n,r,\lambda /2}\leq K2^{\lambda /2}\max_{i,t}\left\{ \left\Vert
x_{i,t}\right\Vert _{r}\vee \left\Vert y_{t}\right\Vert _{r}\right\} =%
\mathcal{K}_{\lambda }\max_{i,t}\left\{ \left\Vert x_{i,t}\right\Vert
_{r}\vee \left\Vert y_{t}\right\Vert _{r}\right\} . \\
\text{\ }&  & & \Phi _{n,\psi _{\mu },\lambda /2}\leq \mathcal{K}_{\lambda
}\sup_{r\geq 2}\frac{1}{r^{\mu }}r^{\mu }=\mathcal{K}_{\lambda }.
\end{align*}

By Theorem 3($ii$) in \cite{ChangChenWu2024}, with their Condition 1 $B_{n}$ 
$=$ $1$, and with $\alpha $ $=$ $\lambda /2$ and $\nu $ $=$ $\mu $,%
\begin{align*}
& & g_{n}& =\frac{\text{\ }\left( \ln p\right) ^{7/6}}{n^{\lambda
/(24+6\lambda )}}+\frac{\Psi _{n,2,\lambda /2}^{1/3}\Psi
_{n,2,0}^{1/3}\left( \ln p\right) ^{2/3}}{n^{\lambda /(24+6\lambda )}}+\frac{%
\Phi _{n,\psi _{\mu },\lambda /2}\left( \ln p\right) ^{1+\mu }}{n^{\lambda
/(8+2\lambda )}} \\
& & & \leq \frac{\text{\ }\left( \ln p\right) ^{7/6}}{n^{\lambda
/(24+6\lambda )}}+\frac{\mathcal{K}_{\lambda }^{1/3}\mathcal{K}%
_{0}^{1/3}\max_{i,t}\left\{ \left\Vert x_{i,t}\right\Vert _{4}\vee
\left\Vert y_{t}\right\Vert _{4}\right\} ^{2/3}\left( \ln p\right) ^{2/3}}{%
n^{\lambda /(24+6\lambda )}}+\frac{\mathcal{K}_{\lambda }\left( \ln p\right)
^{1+\mu }}{n^{\lambda /(8+2\lambda )}} \\
& & & \leq K\max_{i,t}\left\{ \left\Vert x_{i,t}\right\Vert _{4}\vee
\left\Vert y_{t}\right\Vert _{4}\right\} ^{2/3}\frac{\left( \ln p\right)
^{(7/6)\vee (1+\mu )}}{n^{\lambda /(8+2\lambda )}}\rightarrow 0
\end{align*}%
for $\{p\}$ satisfying%
\begin{equation*}
\ln (p)=o\left( \left( \frac{n^{\frac{3\lambda /2}{8+2\lambda }}}{%
\max_{i,t}\left\{ \left\Vert x_{i,t}\right\Vert _{4}\vee \left\Vert
y_{t}\right\Vert _{4}\right\} }\right) ^{\frac{2}{3\left\{ (7/6)\vee (1+\mu
)\right\} }}\right) .
\end{equation*}

Moreover, $\limsup_{r\rightarrow \infty }r^{1/2-1/\alpha
}\limsup_{n\rightarrow \infty }\max_{i,n}\check{\Theta}_{i}^{(q)}$ $<$ $%
\infty $ with $\alpha $ $=$ $2/(2\mu $ $+$ $1)$ follows instantly from $\lim
\sup_{n\rightarrow \infty }||\check{z}||_{n,r}$ $\leq $ $ar^{\mu }$ for
finite $a>0$ and $\mu $ $\in $ $[0,\infty )$ by the proof of Lemma \ref%
{lm:mr_ng_phys}.b. Combined with non-degeneracy it follows $%
\max_{i,t}\{||x_{i,t}||_{4}\vee ||y_{t}||_{4}\}$ $\in $ $(0,\infty )$
uniformly in $n$. Hence $p$ must satisfy%
\begin{equation*}
\ln (p)=o\left( n^{\frac{3\lambda /2}{8+2\lambda }\frac{2}{3\left\{
(7/6)\vee (1+\mu )\right\} }}\right) =o\left( n^{\frac{\lambda }{8+2\lambda }%
\frac{1}{(7/6)\vee (1+\mu )}}\right) .\text{ }\blacksquare
\end{equation*}

The proof of Theorem \ref{thm:mr_gauss_phys} requires the following. See %
\citet[Appendix F]{sm_sig_pred}. Define%
\begin{align}
& & \mathcal{T}_{n}& :=\max_{1\leq i\leq p}\left\vert \sqrt{n}\mathcal{W}%
_{i,n}\hat{\phi}_{i,n}\right\vert \text{ and }\mathcal{\mathring{T}}%
_{n}:=\max_{1\leq i\leq p}\left\vert \frac{1}{\sqrt{n}}\sum_{t=1}^{n}\left[
0,1\right] \mathcal{W}_{i}\mathcal{H}_{i}^{-1}\mathbf{x}_{i,t}v_{t}\right%
\vert   \label{TTG} \\
& & \mathfrak{Z}_{n}& :=\max_{1\leq i\leq p}\left\vert \mathcal{Z}%
_{n}(i)\right\vert .  \notag
\end{align}

\begin{lemma}
\label{lm:PP}For any $\varepsilon $ $>$ $0$, $\sup_{c\geq 0}|\mathbb{P}(%
\mathcal{T}_{n}>c)-\mathbb{P}(\mathcal{\mathring{T}}_{n}>c)|$ $\leq $ $%
\sup_{c\geq 0}\mathbb{P}(c-\varepsilon \leq \mathfrak{Z}_{n}\leq
c+\varepsilon )$ $+$ $\mathbb{P}(\left\vert \mathcal{T}_{n}-\mathcal{%
\mathring{T}}_{n}\right\vert >\varepsilon )$ $+$ $2\rho _{n}.$
\end{lemma}

\noindent \textbf{Proof of Theorem \ref{thm:mr_gauss_phys}}. Using $\mathcal{%
T}_{n}$, $\mathcal{\mathring{T}}_{n}$ and $\mathfrak{Z}_{n}$ from (\ref{TTG}%
), first observe that $\rho _{n}$ $:=$ $\sup_{c\geq 0}|\mathbb{P}(\mathcal{%
\mathring{T}}_{n}$ $>$ $c)$ $-$ $\mathbb{P}(\mathfrak{Z}_{n}$ $>$ $c)|$ $%
\rightarrow $ $0$ by Lemma \ref{lm:mr_gauss_phys}. Moreover, for some $%
\kappa $ $>$ $1/2$%
\begin{align}
& & \left\vert \mathcal{T}_{n}-\mathcal{\mathring{T}}_{n}\right\vert & \leq
\max_{1\leq i\leq p}\left\vert \mathcal{W}_{i,n}\right\vert \max_{1\leq
i\leq p}\left\vert \sqrt{n}\hat{\phi}_{i,n}-\frac{1}{\sqrt{n}}\sum_{t=1}^{n}%
\left[ 0,1\right] \mathcal{W}_{i}\mathcal{H}_{i}^{-1}\mathbf{x}%
_{i,t}v_{t}\right\vert   \label{TT} \\
& & & \text{ \ \ \ \ \ \ }+\max_{1\leq i\leq p}\left\vert \mathcal{W}_{i,n}-%
\mathcal{W}_{i}\right\vert \max_{1\leq i\leq p}\left\vert \frac{1}{\sqrt{n}}%
\sum_{t=1}^{n}\left[ 0,1\right] \mathcal{H}_{i}^{-1}\mathbf{x}%
_{i,t}v_{t}\right\vert   \notag \\
& & & =O_{p}\left( \frac{\left( \ln (p)\right) ^{2}}{\sqrt{n}}\right)
+o_{p}\left( \frac{1}{(\ln (p))^{1+\kappa }}\right) \times O_{p}\left( \ln
(p)\right) =O_{p}\left( \frac{\left( \ln (p)\right) ^{2}}{\sqrt{n}}\vee 
\frac{1}{(\ln (p))^{\kappa }}\right) .  \notag
\end{align}%
The last line exploits Lemma \ref{lm:mr_ng_phys}, the triangle inequality
and supposition $\max_{1\leq i\leq p}|\mathcal{W}_{i,n}$ $-$ $\mathcal{W}%
_{i}|$ $=$ $o_{p}(1/(\ln (p))^{1+\kappa })$, $\max_{1\leq i\leq p}||\mathcal{%
H}_{i}^{-1}||$ $=$ $O(1)$ under Assumption \ref{assum:marg-reg}.c, and by
Lemma B.4 $\max_{1\leq i\leq p}|1/\sqrt{n}\sum_{i=1}^{n}x_{i,t}v_{t}|$ $=$ $%
O_{p}(\ln (p))$ under $H_{0}$.

Now, by the triangle inequality and Lemma \ref{lm:PP} for any $\varepsilon $ 
$>$ $0$%
\begin{align*}
& & \rho _{n}^{\prime }& :=\sup_{c\geq 0}\left\vert \mathbb{P}\left( 
\mathcal{T}_{n}>c\right) -\mathbb{P}\left( \mathfrak{Z}_{n}>c\right)
\right\vert  \\
& & & \leq \rho _{n}+\sup_{c\geq 0}\left\vert \mathbb{P}\left( \mathcal{T}%
_{n}>c\right) -\mathbb{P}\left( \mathcal{\mathring{T}}_{n}>c\right)
\right\vert  \\
& & & \leq 3\rho _{n}+\sup_{c\geq 0}\mathbb{P}\left( c-\varepsilon \leq 
\mathfrak{Z}_{n}\leq c+\varepsilon \right) +\mathbb{P}\left( \left\vert 
\mathcal{T}_{n}-\mathcal{\mathring{T}}_{n}\right\vert >\varepsilon \right) .
\end{align*}%
Now set $\varepsilon $ $=$ $\varepsilon _{n}$ $=$ $1/(\ln (p))^{\beta }$ for 
$\beta $ $\in $ $(1/2,\kappa )$ to yield by (\ref{TT})%
\begin{align*}
& & \frac{\left\vert \mathcal{T}_{n}-\mathcal{\mathring{T}}_{n}\right\vert }{%
\varepsilon _{n}}& =O_{p}\left( \frac{\left( \ln (p)\right) ^{\beta }\left(
\ln (p)\right) ^{2+\kappa }}{\sqrt{n}}\vee \frac{\left( \ln (p)\right)
^{\beta }}{(\ln (p))^{\kappa }}\right)  \\
& & & =O_{p}\left( \frac{\left( \ln (p)\right) ^{2+\beta +\kappa }}{\sqrt{n}}%
\vee \frac{1}{(\ln (p))^{\kappa -\beta }}\right) =o_{p}\left( 1\right) \text{
when }\ln (p)=o\left( n^{\frac{1}{4(1+\kappa )}}\right) ,
\end{align*}%
thus $\mathbb{P}(|\mathcal{T}_{n}$ $-$ $\mathcal{\mathring{T}}_{n}|$ $>$ $%
\varepsilon _{n})$ $\rightarrow $ $0$. Moreover, from Nazarov's
anti-concentration inequality and $\beta $ $>$ $1/2$ we have for some $%
\mathcal{C}$ $>$ $0$%
\begin{equation*}
\sup_{c\geq 0}\mathbb{P}\left( c-\varepsilon _{n}\leq \mathfrak{Z}_{n}\leq
c+\varepsilon _{n}\right) \leq \mathcal{C}\varepsilon _{n}\sqrt{\ln (p)}=%
\mathcal{C}\frac{1}{\left( \ln (p)\right) ^{\beta -1/2}}=o(1).
\end{equation*}%
Now invoke Lemma \ref{lm:mr_gauss_phys} to conclude $\rho _{n}^{\prime }$ $%
\leq $ $3\rho _{n}$ $+\mathcal{C}/(\ln (p))^{\beta -1/2}$ $+$ $\mathbb{P}(|%
\mathcal{T}_{n}$ $-$ $\mathcal{\mathring{T}}_{n}|$ $>$ $\varepsilon _{n})$ $%
\rightarrow $ $0$ provided $\ln (p)=o(n^{\frac{1}{4(1+\kappa )}\wedge \{%
\frac{\lambda }{8+2\lambda }\frac{1}{(7/6)\vee (1+\mu )}\}}).$ $\blacksquare
\medskip \newline
\medskip $\textbf{Proof of Theorem \ref{thm:mr:boot}.} Define $\mathcal{F}%
_{n}(c)$ $:=$ $\mathbb{P}(\mathfrak{\tilde{Z}}_{n}$ $\leq $ $c)$. By
construction $\hat{p}_{n,\mathcal{M}_{n}}$ $=$ $1$ $-$ $\mathbb{P}(\mathcal{%
\tilde{T}}_{n}$ $\leq $ $\mathcal{\hat{T}}_{n}$ $\mid $ $\mathfrak{S}_{n})$ $%
+$ $o_{p}\left( 1\right) $, and by Lemma \ref{lm:mr_boot_Tn} $\mathbb{P}(%
\mathcal{\tilde{T}}_{n}$ $\leq $ $c$ $\mid $ $\mathfrak{S}_{n})$ $=$ $%
\mathcal{F}_{n}(c)$ $+$ $o_{p}(1)$ uniformly on $[0,\infty )$ when $\ln (p)$ 
$=$ $o(n^{\frac{1}{4(1+\kappa )}}$ $\wedge $ $n^{\frac{1-3b}{2}})$ given $%
b_{n}$ $\approx $ $n^{b}$. Thus $\hat{p}_{n,\mathcal{M}_{n}}=1$ $-$ $%
\mathcal{F}_{n}(\mathcal{\hat{T}}_{n})$ $+$ $o_{p}(1)$.\medskip \newline
\textbf{Claim (a).} But $\sup_{c\geq 0}|\mathbb{P}(\mathcal{T}_{n}$ $\leq $ $%
c)$ $-$ $\mathcal{F}_{n}(c)|$ $=$ $o(1)$ under $H_{0}$ by Theorem \ref%
{thm:mr_gauss_phys} when $\ln (p)$ $=$ $o(n^{s(\mu ,\kappa ,\lambda )})$
with $s(\mu ,\kappa ,\lambda )$ in (\ref{s(b,lam)}). Thus since by
continuity $\{\mathcal{F}_{n}(\mathcal{\hat{T}}_{n})$ $\leq $ $\alpha \}$ $=$
$\{\mathcal{\hat{T}}_{n}$ $\leq $ $\mathcal{F}_{n}^{-1}(\alpha )\}$ it
follows uniformly in $\alpha $ $\in $ $[0,1]$ 
\begin{equation*}
\mathbb{P}\left( \mathcal{F}_{n}(\mathcal{\hat{T}}_{n})\leq \alpha \right) =%
\mathbb{P}\left( \mathcal{\hat{T}}_{n}\leq \mathcal{F}_{n}^{-1}(\alpha
)\right) =\mathcal{F}_{n}\left( \mathcal{F}_{n}^{-1}(\alpha )\right)
+o\left( 1\right) =\alpha +o\left( 1\right) 
\end{equation*}%
yielding $\sup_{u\geq u_{0}}|\mathbb{P}(\mathcal{F}_{n}(\mathcal{\hat{T}}%
_{n})$ $\leq $ $\alpha )$ $-$ $\alpha |$ $=$ $o(1)$. Use Slutsky's theorem
to conclude%
\begin{equation*}
\mathbb{P}\left( \hat{p}_{n,\mathcal{M}_{n}}\leq \alpha \right) =\mathbb{P}%
\left( \mathcal{F}_{n}(\mathcal{\hat{T}}_{n})\geq 1-\alpha +o_{p}(1)\right) =%
\mathbb{P}\left( \mathcal{\hat{T}}_{n}\geq \mathcal{F}_{n}^{-1}(1-\alpha
)\right) \rightarrow \alpha .
\end{equation*}%
\textbf{Claim (b).} Let $H_{1}$ hold and assume $\boldsymbol{x}_{i,t}$ $=$ $%
x_{i,t}$ ($\delta _{i}$ $=$ $0$ is known and the constant is dropped).
Define the test statistic under the null $\mathcal{\hat{T}}_{n}^{(0)}$ $:=$ $%
\max_{1\leq i\leq p}|\widehat{\mathcal{H}}_{i,n}^{-1}n^{-1/2}%
\sum_{t=1}^{n}x_{i,t}v_{i,t}|$. Then $\mathcal{\hat{T}}_{n}$ $=$ $\sqrt{n}%
\max_{1\leq i\leq p}|\phi _{i}+\widehat{\mathcal{H}}_{i,n}^{-1}n^{-1}%
\sum_{t=1}^{n}x_{i,t}v_{i,t}|$. Define $\phi _{n}^{\ast }$ $:=$ $\max_{1\leq
i\leq p}|\phi _{i}|$ $>$ $0$. Under $H_{1}$, $\lim \inf_{n\rightarrow \infty
}\phi _{n}^{\ast }$ $>$ $0$. By the reverse triangle inequality, 
\begin{equation*}
\mathcal{\hat{T}}_{n}=\sqrt{n}\max_{1\leq i\leq p}\left\vert \phi _{i}+%
\widehat{\mathcal{H}}_{i,n}^{-1}\frac{1}{n}\sum_{t=1}^{n}x_{i,t}v_{i,t}%
\right\vert \geq \sqrt{n}\phi _{n}^{\ast }-\max_{1\leq i\leq p}\left\vert 
\widehat{\mathcal{H}}_{i,n}^{-1}\frac{1}{\sqrt{n}}%
\sum_{t=1}^{n}x_{i,t}v_{i,t}\right\vert =\sqrt{n}\phi _{n}^{\ast }-\mathcal{%
\hat{T}}_{n}^{(0)}.
\end{equation*}%
By Theorem \ref{thm:mr_gauss_phys}, $\mathcal{\hat{T}}_{n}^{(0)}$ $=$ $O_{p}(%
\sqrt{\ln (p)})$. Since $\ln (p)=o(\sqrt{n})$ we have $\mathcal{\hat{T}}%
_{n}^{(0)}/\sqrt{n}$ $=$ $o_{p}(1)$, which yields $\mathcal{\hat{T}}_{n}/%
\sqrt{n}$ $\geq $ $\phi _{n}^{\ast }$ $-$ $o_{p}(1)$ $=$ $\phi _{n}^{\ast }(1
$ $+$ $o_{p}(1))$. Consequently, $\mathcal{\hat{T}}_{n}\overset{p}{%
\longrightarrow }\infty $ and therefore $\hat{p}_{n,\mathcal{M}_{n}}=1$ $-$ $%
\mathcal{F}_{n}(\mathcal{\hat{T}}_{n})$ $+$ $o_{p}(1)$ $\overset{p}{%
\rightarrow }$ $0$ by the mapping theorem. $\blacksquare $\medskip \newline
\textbf{Proof of Theorem \ref{thm:mr:boot_local}.} The proof mimics %
\citet[proof of Theorem 3.2]{Hill_2025_maxtest}. $\blacksquare $.

\setstretch{1.75} 
\bibliographystyle{apa}
\bibliography{refs_sig_pred}

\begin{thebibliography}{}

\bibitem[\protect\astroncite{Andrews}{1988}]{Andrews1988}
Andrews, D. W.~K. (1988).
\newblock Laws of large numbers for dependent non-identically distributed
  random variables.
\newblock {\em Economet. Theory}, 4:458--467.

\bibitem[\protect\astroncite{Andrews}{1999}]{Andrews1999}
Andrews, D. W.~K. (1999).
\newblock Estimation when a parameter is on a boundary.
\newblock {\em Econometrica}, 67:1341--1383.

\bibitem[\protect\astroncite{Andrews and Cheng}{2012}]{Andrews_Cheng_2012}
Andrews, D. W.~K. and Cheng, X. (2012).
\newblock Estimation and inference with weak, semi-strong and strong
  identification.
\newblock {\em Econometrica}, 80:2153–2211.

\bibitem[\protect\astroncite{Andrews and
  Ploberger}{1994}]{AndrewsPloberger1994}
Andrews, D. W.~K. and Ploberger, W. (1994).
\newblock Optimal tests when a nuisance parameter is present only under the
  alternative.
\newblock {\em Econometrica}, 62:1383--1414.

\bibitem[\protect\astroncite{Angrist and Hahn}{2004}]{AngristHahn2004}
Angrist, J. and Hahn, J. (2004).
\newblock When to control for covariates? panel asymptotics for estimates of
  treatment effects.
\newblock {\em Review Econom. Statist.}, 86:58--72.

\bibitem[\protect\astroncite{Basrak et~al.}{2002}]{Basrak_etal2002}
Basrak, B., Davis, R.~A., and Mikosch, T. (2002).
\newblock Regular variation of garch processes.
\newblock {\em Stoch. Process. Appl.}, 99:95--115.

\bibitem[\protect\astroncite{Belloni
  et~al.}{2014}]{BelloniChernozhukovHansen2014}
Belloni, A., Chernozhukov, V., and Hansen, C. (2014).
\newblock High-dimensional methods and inference on structural and treatment
  effects.
\newblock {\em J. Econom. Perspect.}, 28:29--50.

\bibitem[\protect\astroncite{Benjamini and
  Hochberg}{1995}]{BenjaminiHochberg1995}
Benjamini, Y. and Hochberg, Y. (1995).
\newblock Controlling the false discovery rate: A practical and powerful
  approach to multiple testing.
\newblock {\em J. R. Stat. Soc. Ser. B}, 57:289–300.

\bibitem[\protect\astroncite{Berk et~al.}{2013}]{Berk_etal2013}
Berk, R., Brown, L.~D., Buja, A., Zhang, K., and Zhao, L. (2013).
\newblock Valid post-selection inference.
\newblock {\em Ann. Statist.}, 41:802--837.

\bibitem[\protect\astroncite{Bose}{1988}]{Bose1988}
Bose, A. (1988).
\newblock Edgewordth correction by bootstrap in autoregressions.
\newblock {\em Ann. Statist.}, 16:1709--1722.

\bibitem[\protect\astroncite{Bougerol and Picard}{1992}]{BougerolPicard1992}
Bougerol, P. and Picard, N. (1992).
\newblock Strict stationarity of generalized autoregressive processes.
\newblock {\em Annals of Probability}, 20:1714--1730.

\bibitem[\protect\astroncite{Buhlmann and van~de
  Geer}{2011}]{BuhlmannVanDeGeer2011}
Buhlmann, P. and van~de Geer, S. (2011).
\newblock {\em Statistics for High-Dimensional Data}.
\newblock Springer, Berlin.

\bibitem[\protect\astroncite{Cai and Jiang}{2012}]{CaiJiang2012}
Cai, T.~T. and Jiang, T. (2012).
\newblock Phase transition in limiting distributions of coherence of
  high-dimensional random matrices.
\newblock {\em J. Multivariate Anal.}, 107:24–39.

\bibitem[\protect\astroncite{Chang et~al.}{2024}]{ChangChenWu2024}
Chang, J., Chen, X., and Wu, M. (2024).
\newblock Central limit theorems for high dimensional dependent data.
\newblock {\em Bernoulli}, 30:712--742.

\bibitem[\protect\astroncite{Chang et~al.}{2023}]{ChangJiangShao2023}
Chang, J., Jiang, Q., and Shao, X. (2023).
\newblock Testing the martingale difference hypothesis in high dimension.
\newblock {\em J. Econometrics}, 235:972--1000.

\bibitem[\protect\astroncite{Chang et~al.}{2013}]{ChangTangWu2013}
Chang, J., Tang, C., and Wu, Y. (2013).
\newblock Marginal empirical likelihood and sure independence feature
  screening.
\newblock {\em Ann. Statist.}, 41:2123--2148.

\bibitem[\protect\astroncite{Cheng}{2015}]{Cheng2015}
Cheng, X. (2015).
\newblock Robust inference in nonlinear models with mixed indentification
  strength.
\newblock {\em J. Econometrics}, 189:207--228.

\bibitem[\protect\astroncite{Chernozhukov et~al.}{2013}]{Chernozhukov_etal2013}
Chernozhukov, V., Chetverikov, D., and Kato, K. (2013).
\newblock Gaussian approximations and multiplier bootstrap for maxima of sums
  of high-dimensional random vectors.
\newblock {\em Ann. Statist.}, 41:2786--2819.

\bibitem[\protect\astroncite{Chernozhukov et~al.}{2015}]{Chernozhukov_etal2015}
Chernozhukov, V., Chetverikov, D., and Kato, K. (2015).
\newblock Comparison and anti-concentration bounds for maxima of {G}aussian
  random vectors.
\newblock {\em Probab. Theory Related Fields}, 162:47--70.

\bibitem[\protect\astroncite{Correia}{2016}]{Correia2016}
Correia, S. (2016).
\newblock A feasible estimator for linear models with multi-way fixed effects.
\newblock Unpublished manuscipt (scorreia.com/research/hdfe.pdf).

\bibitem[\protect\astroncite{Dudoit et~al.}{2003}]{Dudoit_etal2003}
Dudoit, S., Shaffer, J.~P., and Boldrick, J.~C. (2003).
\newblock Multiple hypothesis testing in microarray experiments.
\newblock {\em Statist. Sci.}, 18:71--103.

\bibitem[\protect\astroncite{Efron}{2006}]{Efron2006}
Efron, B. (2006).
\newblock Large-scale simultaneous hypothesis testing: The choice of a null
  hypothesis.
\newblock {\em J. Amer. Statist. Assoc.}, 99:96–104.

\bibitem[\protect\astroncite{Fan and Li}{2006}]{FanLi2006}
Fan, J. and Li, R. (2006).
\newblock Statistical challenges with high dimensionality: Feature selection in
  knowledge discovery.
\newblock In Sanz-Sole, M., Soria, J., Varona, J.~L., and Verdera, J., editors,
  {\em Proc. Internat. Congress of Mathematicians}, volume III, pages 595--622,
  Zurich. European Mathematical Society.

\bibitem[\protect\astroncite{Fan and Lv}{2008}]{FanLv2008}
Fan, J. and Lv, J. (2008).
\newblock Sure independence screening for ultrahigh dimensional feature space.
\newblock {\em J. R. Stat. Soc. Ser. B}, 70:849--911.

\bibitem[\protect\astroncite{Fan et~al.}{2011}]{FanLvQi2011}
Fan, J., Lv, J., and Qi, .~L. (2011).
\newblock Sparse high-dimensional models in economics.
\newblock {\em Annual Econ. Rev.}, 3:291--317.

\bibitem[\protect\astroncite{Genovese et~al.}{2012}]{Genovese_etal_2012}
Genovese, C.~R., Jin, J., Wasserman, L., and Yao, Z. (2012).
\newblock A comparison of the lasso and marginal regression.
\newblock {\em J. Mach. Learning Research}, 13:2107--2143.

\bibitem[\protect\astroncite{Giannone et~al.}{2021}]{Giannone_etal2021}
Giannone, D., Lenza, M., and Primiceri, G.~E. (2021).
\newblock Economic predictions with big data: The illusion of sparsity.
\newblock Technical Report 2542, European Central Bank.

\bibitem[\protect\astroncite{Hill}{2021}]{Hill2021}
Hill, J.~B. (2021).
\newblock Weak identification robust wild bootstrap applied to a consistent
  model specification test.
\newblock {\em Economet. Theory}, 37:409--463.

\bibitem[\protect\astroncite{Hill}{2025a}]{Hill_2025_mixg}
Hill, J.~B. (2025a).
\newblock Mixingale and physical dependence equality with applications.
\newblock {\em Statist. Probab. Letters}, 221:110380.

\bibitem[\protect\astroncite{Hill}{2025b}]{Hill_2025_maxtest}
Hill, J.~B. (2025b).
\newblock Testing many zero restrictions in a high dimensional linear
  regression setting.
\newblock {\em J. Bus. Econom. Statist.}, 43:55--67.

\bibitem[\protect\astroncite{Hill}{2026}]{sm_sig_pred}
Hill, J.~B. (2026).
\newblock Supplemental material for ``{A} high dimensional wild bootstrap
  max-test for detecting the presence of significant predictors".
\newblock Dept. of Economics, University of North Carolina - Chapel Hill.

\bibitem[\protect\astroncite{Hill and Li}{2025}]{HillLi2025}
Hill, J.~B. and Li, T. (2025).
\newblock A bootstrapped test of covariance stationarity based on orthonormal
  transformations.
\newblock {\em Bernoulli}, 31:1527--1551.

\bibitem[\protect\astroncite{Hill and Motegi}{2020}]{HillMotegi2020}
Hill, J.~B. and Motegi, K. (2020).
\newblock A max-correlation white noise for weakly dependent time series.
\newblock {\em Economet. Theory}, 36:907--960.

\bibitem[\protect\astroncite{Huang et~al.}{2019}]{Huang_etal2019}
Huang, T.-J., McKeague, I.~W., and Qian, M. (2019).
\newblock Marginal screening for high-dimensional predictors of survival
  outcomes.
\newblock {\em Stat. Sin.}, 29:2105--2139.

\bibitem[\protect\astroncite{Jiang}{2004}]{Jiang2004}
Jiang, T. (2004).
\newblock The asymptotic distributions of the largest entries of sample
  correlation matrices.
\newblock {\em Ann. Appl. Probab.}, 14:865–880.

\bibitem[\protect\astroncite{Koles{\'a}r et~al.}{2024}]{Kolesar_etal_2024}
Koles{\'a}r, M., M{\"u}ller, U.~K., and Roeslgaard, S.~T. (2024).
\newblock The fragility of sparsity.
\newblock Dept. of Economics, Princeton University.

\bibitem[\protect\astroncite{Laber and Murphy}{2011}]{LaberMurphy2011}
Laber, E. and Murphy, S.~A. (2011).
\newblock Adaptive confidence intervals for the test error in classification.
\newblock {\em J. Amer. Statist. Assoc.}, 106:904--913.

\bibitem[\protect\astroncite{Leeb and P\"{o}tscher}{2006}]{LeebPotscher2006}
Leeb, H. and P\"{o}tscher, B.~M. (2006).
\newblock Can one estimate the conditional distribution of post-model-selection
  estimators.
\newblock {\em Ann. Statist.}, 34:2554--2591.

\bibitem[\protect\astroncite{Li et~al.}{2009}]{LiLiuRosalsky2009}
Li, D., Liu, W., and Rosalsky, A. (2009).
\newblock Necessary and sufficient conditions for the asymptotic distribution
  of the largest entry of a sample correlation matrix.
\newblock {\em Probab. Theory Related Fields}, 148:5--35.

\bibitem[\protect\astroncite{Liu}{1988}]{Liu1988}
Liu, R.~Y. (1988).
\newblock Bootstrap procedures under some non-i.i.d. models.
\newblock {\em Ann. Statist.}, 16:1696--1708.

\bibitem[\protect\astroncite{Liu et~al.}{2008}]{LiuLinShao2008}
Liu, W., Lin, Z., and Shao, Q. (2008).
\newblock The asymptotic distribution and berry–esseen bound of a new test
  for independence in high dimension with an application to stochastic
  optimization.
\newblock {\em Ann. Appl. Probab.}, 18:2337--2366.

\bibitem[\protect\astroncite{Ljung and Box}{1978}]{LjungBox1978}
Ljung, G.~M. and Box, G. E.~P. (1978).
\newblock On a measure of lack of fit in time series models.
\newblock {\em Biometrika}, 65:297--303.

\bibitem[\protect\astroncite{McCloskey}{2017}]{McCloskey2017}
McCloskey, A. (2017).
\newblock Bonferroni-based size-correction for nonstandard testing problems.
\newblock {\em J. Econometrics}, 200:17--35.

\bibitem[\protect\astroncite{McCloskey}{2020}]{McCloskey2020}
McCloskey, A. (2020).
\newblock Asymptotically uniform tests after consistent model selection in the
  linear regression model.
\newblock {\em J. Bus. Econom. Statist.}, 38:810--825.

\bibitem[\protect\astroncite{McKeague and Qian}{2015}]{McKeague_Qian_2015}
McKeague, I. and Qian, M. (2015).
\newblock An adaptive resampling test for detecting the presence of significant
  predictors.
\newblock {\em J. Amer. Statist. Assoc.}, 110:1422--1433.

\bibitem[\protect\astroncite{McKeague and Zhang}{2022}]{McKeagueZhang2022}
McKeague, I. and Zhang, I. (2022).
\newblock Significance testing for canonical correlation analysis in high
  dimensions.
\newblock {\em Biometrika}, 109:1076--1083.

\bibitem[\protect\astroncite{McLeish}{1975}]{McLeish1975}
McLeish, D.~L. (1975).
\newblock A maximal inequality and dependent strong laws.
\newblock {\em Ann. Probab.}, 3:829--839.

\bibitem[\protect\astroncite{Rio}{2017}]{Rio2017}
Rio, E. (2017).
\newblock {\em Asymptotic Theory of Weakly Dependent Random Processes}.
\newblock Springer.

\bibitem[\protect\astroncite{Sawa}{1978}]{Sawa1978}
Sawa, T. (1978).
\newblock Iinformation criteria for discriminating among alternative regression
  models.
\newblock {\em Econometrica}, 46:1273--1291.

\bibitem[\protect\astroncite{Shao and Zhou}{2014}]{ShaoZhou2014}
Shao, Q.-M. and Zhou, W.-X. (2014).
\newblock Necessary and sufficient conditions for the asymptotic distributions
  of coherence of ultra-high dimensional random matrices.
\newblock {\em Ann. Probab.}, 42:623--648.

\bibitem[\protect\astroncite{Shao}{2011}]{Shao2011_JoE}
Shao, X. (2011).
\newblock A bootstrap-assisted spectral test of white noise under unknown
  dependence.
\newblock {\em J. Econometrics}, 162:213--224.

\bibitem[\protect\astroncite{Tang et~al.}{2018}]{TangWangBarut2018}
Tang, Y., Wang, H.~J., , and Barut, E. (2018).
\newblock Testing for the presence of significant covariates through
  conditional marginal regression.
\newblock {\em Biometrika}, 105:57--71.

\bibitem[\protect\astroncite{Vershynin}{2018}]{Vershynin2018}
Vershynin, R. (2018).
\newblock {\em High-Dimensional Probability}.
\newblock Cambridge University Press, Cambridge, UK.

\bibitem[\protect\astroncite{White}{1982}]{White1982}
White, H. (1982).
\newblock Maximum likelihood estimation of misspecified models.
\newblock {\em Econometrica}, 50:1--25.

\bibitem[\protect\astroncite{Wu}{2005}]{Wu2005}
Wu, W.~B. (2005).
\newblock Nonlinear system theory: Another look at dependence.
\newblock {\em Proc. Natl. Acad. Sci.}, 102:14150--14154.

\bibitem[\protect\astroncite{Wu and Min}{2005}]{WuLin2005}
Wu, W.~B. and Min, M. (2005).
\newblock On linear processes with dependent innovations.
\newblock {\em Stochastic Process. Appl.}, 115:939--958.

\bibitem[\protect\astroncite{Wu and Wu}{2016}]{WuWu2016}
Wu, W.~B. and Wu, Y.~N. (2016).
\newblock Performance bounds for parameter estimates of high-dimensional linear
  models with correlated errors.
\newblock {\em Electron. J. Statist.}, 10:352--379.

\bibitem[\protect\astroncite{Zhang and Laber}{2015}]{ZhangLaber2015}
Zhang, Y. and Laber, E.~B. (2015).
\newblock Comment: An adaptive resampling test fordetecting the presence of
  signicant predictors.
\newblock {\em J. Amer. Statist. Assoc.}, 110:1451--1454.

\end{thebibliography}
\singlespacing\setstretch{1} \clearpage

\clearpage

\end{document}